\documentclass{article}

\usepackage[T1]{fontenc}
\usepackage{amsmath}
\usepackage{amsfonts}
\usepackage{amssymb}
\usepackage{mathdots}
\usepackage{mathtools}
\usepackage{amsthm}
\usepackage[margin=1in]{geometry}
\usepackage[subrefformat=simple, labelformat=simple]{subcaption}
\usepackage{tikz}
\usepackage{tikz-cd}
\usepackage{hyperref}
\usepackage[nameinlink, compress, capitalise]{cleveref}
\usepackage[square]{natbib}
\usepackage{authblk}
\usepackage[shortlabels]{enumitem}

\allowdisplaybreaks[2]
\mathtoolsset{mathic=true}

\usetikzlibrary{calc, fit, decorations.markings, shapes.geometric}
\tikzset{3D/.style={x={(0.6cm,0.6cm)}, y={(-1.0cm,0.3cm)}, z={(0cm,1cm)}}, %
vertex/.style={circle, inner sep=1pt, fill}, %
missing vertex/.style={vertex, inner sep=2pt, fill=white, draw=black!30, thin}, %
edge/.style={semithick, line join=round, line cap=round}, %
missing edge/.style={edge, dashed}, %
sketch edge/.style={thin, densely dotted, draw opacity=0.6}, %
hyperplane/.style={edge, thin, dashed}, %
face/.style={fill=black!15, fill opacity=0.7, edge}, %
baseline={($ (current bounding box.west) - (0,1ex) $)} %
}

\tikzcdset{every diagram/.append style={column sep=small, row sep=small}, cells={font=\everymath\expandafter{\the\everymath\displaystyle}}}

\renewcommand{\emptyset}{\varnothing}
\renewcommand{\epsilon}{\varepsilon}

\renewcommand{\tilde}{\widetilde}

\DeclarePairedDelimiter{\abs}{\lvert}{\rvert}

\newcommand{\Squelta}{{\mathord{{\tikz[baseline=0.5pt]{\fill [even odd rule] (0,0) [rounded corners=0.15pt] rectangle (0.68em,0.68em) (0.04em,0.07em) [sharp corners] rectangle (0.68em-0.09em,0.68em-0.04em); \useasboundingbox (-0.08em,0) rectangle (0.68em+0.08em,0.68em);}}}}}
\DeclareMathOperator{\skel}{Skel}
\DeclareMathOperator{\lk}{link}
\DeclareMathOperator{\slk}{s-link}
\DeclareMathOperator{\st}{star}
\DeclareMathOperator{\bary}{bary}
\DeclareMathOperator{\cross}{Cross}
\DeclareMathOperator{\cone}{Cone}
\DeclareMathOperator{\im}{im}
\DeclareMathOperator{\rank}{rank}

\theoremstyle{plain}
\newtheorem{thm}{Theorem}[section]
\newtheorem{lma}[thm]{Lemma}
\newtheorem{prop}[thm]{Proposition}
\newtheorem{crl}[thm]{Corollary}

\theoremstyle{definition}

\theoremstyle{remark}
\newtheorem{rmk}[thm]{Remark}

\Crefname{thm}{Theorem}{Theorems}
\Crefname{figure}{Figure}{Figures}

\title{Topology of complements of skeletons}
\author[1]{Rowan Rowlands}
\affil[1]{University of Washington, Seattle, WA 98195 US. Email: \url{rowanr@uw.edu}}
\date{}

\begin{document}

\maketitle

\begin{abstract}
Given a polytopal complex $X$, we examine the topological complement of its $k$-skeleton. We construct a long exact sequence relating the homologies of the skeleton complements and links of faces in $X$, and using this long exact sequence, we obtain characterisations of Cohen--Macaulay and Leray complexes, stacked balls, and neighbourly spheres in terms of their skeleton complements. We also apply these results to CAT(0) cubical complexes, and find new similarities between such a complex and an associated simplicial complex, the crossing complex.
\end{abstract}

\section{Introduction} \label{sec:introduction}

Polytopal complexes are important objects in topology and combinatorics, which include simplicial complexes, cubical complexes and polytopes. Their geometric realisations provide examples of a wide array of topological spaces, and much research has gone into studying the interplay between their combinatorial and topological aspects.

One important feature of a polytopal complex is its \emph{$k$-skeleton}, the set of faces of the complex of dimension less than or equal to $k$. Skeletons act as a framework which the high-dimensional faces are attached to, so studying the structure of a skeleton can reveal a lot about a complex. For example, \citet{art:Bayer} surveys many results about objects that can be reconstructed from their $k$-skeletons for certain values of $k$. Skeletons also play an important role in defining cellular homology, and in important topological results such as Poincar\'e duality.

In this paper, however, we aim to approach this topic from the other direction, starting from the higher-dimensional faces instead of the low-dimensional ones. We define the \emph{$k$th co-skeleton} of a complex to be the set of faces of dimension \emph{higher} than $k$. We discover that there is a strong relationship between co-skeletons and \emph{links}: a link of a face in a polytopal complex captures the local structure of the complex around that face, so in a sense, the co-skeletons give us a ``global'' summary of the ``local'' information of the complex.

\Cref{sec:preliminaries} of this paper sets out the basic definitions we will use. In \cref{sec:general}, we examine some initial facts about co-skeletons of arbitrary complexes, and their connections with various forms of duality (see \cref{thm:polytope-duality,thm:Alexander-duality-for-coskeletons}). In \cref{sec:LES}, we construct a long exact sequence relating the $k$th and $(k-1)$th co-skeletons and the links of $k$-dimensional faces (\cref{thm:LES}). In \cref{sec:families}, we use this long exact sequence to study some families of complexes defined by properties of links --- specifically, Cohen--Macaulay complexes, Leray complexes, and stacked balls --- and give characterisations of these families in terms of the homology of their co-skeletons (\cref{thm:CM-characterisation,thm:Leray-characterisation,thm:stacked-characterisation}). Finally, \cref{sec:CAT(0)} examines cubical complexes, particularly CAT(0) cubical complexes: we examine the ``crossing complex'' defined by \citet{art:Rowlands-CAT0}, and show that a CAT(0) cubical complex has one of the properties discussed in \cref{sec:families} if and only if its crossing complex shares that property (\cref{thm:CAT(0)-property-iff-crossing}).

\subsection*{Acknowledgments}

Many thanks to Isabella Novik for her multitudinous helpful suggestions on drafts. Thanks also to Martina Juhnke-Kubitzke for her insightful comments.

The author was partially supported by a graduate fellowship from NSF grant DMS-1953815.

\section{Preliminaries} \label{sec:preliminaries}

We begin with some definitions. Readers familiar with polytopal complexes and topology may skip most of this section, but beware that we give slightly non-standard definitions in a couple of places, specifically for geometric realisations and links.

\subsection{Polytopal complexes}

A \emph{polytopal complex} $X$ is a collection of polytopes in Euclidean space $\mathbb R^N$ with the following properties:
\begin{itemize}
	\item If $\sigma$ is in $X$ and $\tau$ is a face of $\sigma$, then $\tau$ is in $X$, and
	\item If $\sigma$ and $\sigma'$ are polytopes in $X$, then $\sigma \cap \sigma'$ is a face of each (possibly the empty face).
\end{itemize}
See \cref{fig:running-example} for a small example. In this paper, we will only consider finite polytopal complexes.

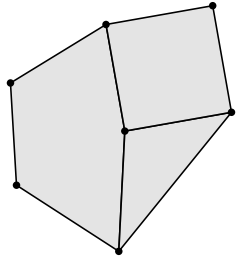
\begin{figure}
\centering
\begin{tikzpicture}[scale=0.8]

\node (v1) [coordinate] at (0,0) {};
\node (v2) [coordinate] at (100:1.8) {};
\node (v3) [coordinate] at (-1.9,0.8) {};
\node (v4) [coordinate] at (-1.8,-0.9) {};
\node (v5) [coordinate] at (-0.1,-2.0) {};
\node (v6) [coordinate] at (10:1.8) {};
\node (v7) [coordinate] at ($(v6) + (100:1.8)$) {};

\foreach \i/\j in {1/2, 2/3, 3/4, 4/5, 1/5, 1/6, 5/6, 2/7, 6/7}
	\node (v\i\j) [coordinate] at ($(v\i)!0.5!(v\j)$) {};
\node (v12345) [coordinate] at ($1/5*(v1)+1/5*(v2)+1/5*(v3)+1/5*(v4)+1/5*(v5)$) {};
\node (v156) [coordinate] at ($1/3*(v1)+1/3*(v5)+1/3*(v6)$) {};
\node (v1267) [coordinate] at ($1/4*(v1)+1/4*(v2)+1/4*(v6)+1/4*(v7)$) {};

\filldraw [face] (v1)--(v2)--(v3)--(v4)--(v5)--cycle;
\filldraw [face] (v1)--(v5)--(v6)--cycle;
\filldraw [face] (v1)--(v2)--(v7)--(v6)--cycle;

\foreach \i in {1,...,7}
	\node [vertex] at (v\i) {};

\end{tikzpicture}
\caption{An example of a polytopal complex} \label{fig:running-example}
\end{figure}

If every polytope in $X$ is a simplex, then we say that $\Delta \coloneqq X$ is a \emph{(geometric) simplicial complex}. If every polytope is a cube (that is, a polytope combinatorially equivalent to $[0,1]^i$ for some dimension $i \geq 0$), then $\Squelta \coloneqq X$ is a \emph{cubical complex}.

An \emph{abstract simplicial complex} $\Delta$ is a poset isomorphic to the poset of faces of a geometric simplicial complex, ordered by inclusion. Equivalently, it is a collection of subsets of some finite set, with the property that $\sigma \in \Delta$ and $\tau \subseteq \sigma$ implies $\tau \in \Delta$. The dimension of a face $\sigma$ is $\dim \sigma \coloneq \# \sigma - 1$.

In any of these types of complex, faces of dimension $0$ and $1$ are called \emph{vertices} and \emph{edges} respectively, and a maximal face (by inclusion) is called a \emph{facet}. If all facets of a complex have the same dimension, the complex is \emph{pure}. The dimension of the complex is the maximum dimension of its faces. The number of faces of dimension $k$ in a complex $X$ is denoted $f_k(X)$.

If $\sigma$ is a polytope, $\abs{\sigma}$ will denote its relative interior. If $S$ is a collection of faces in a polytopal complex, then their \emph{geometric realisation} $\abs{S}$ is the union of the relative interiors of the polytopes in $S$. This is slightly different from the usual definition of a geometric realisation --- if $\sigma$ is in $S$, we do not include the boundary of $\sigma$ in the geometric realisation $\abs{S}$ unless those boundary faces are also part of $S$, unlike other definitions. We will take care to distinguish between $S$ as a set of polytopes and $\abs{S}$ as a topological space.

Suppose $Y$ is a subset of a polytopal complex $X$. We write $X \setminus Y$ to denote the difference of sets, so $X \setminus Y$ is the set of faces of $X$ that are not in $Y$ (which is not generally a polytopal complex). If $Y$ is a polytopal subcomplex of $X$, then $X - Y$ will denote the polytopal complex consisting of all faces of $X$ that do not intersect any faces of $Y$. We will reserve ``$-$'' to denote this ``combinatorial'' deletion, and use ``$\setminus$'' to denote deletion of sets or topological spaces. Note that $\abs{X \setminus Y} = \abs{X} \setminus \abs{Y}$, but the space $\abs{X - Y}$ is not in general the same; see \cref{fig:deletion}, for an example. However, these spaces are sometimes related by the following lemma:
\begin{lma}[{\citep[Lemma~70.1]{book:Munkres-EAT}}] \label{thm:deletion-homotopy-equivalence}
If $\Delta$ is a simplicial complex and $\Lambda$ is an induced subcomplex (in other words, every face of $\Delta$ whose vertices are contained in the vertex set of $\Lambda$ is a face of $\Lambda$), then $\abs{\Delta - \Lambda}$ and $\abs{\Delta} \setminus \abs{\Lambda}$ are homotopy equivalent.
\end{lma}

\begin{figure}
\centering
\begin{subfigure}{0.3\textwidth}
\centering
\begin{tikzpicture}[yscale=0.87, baseline=-1.4cm]

\foreach \x/\y [count=\i] in {-0.5/1, 0.5/1, 1.5/1, 0/0, 1/0, 2/0, 1.5/-1, 2.5/-1}
	\coordinate (v\i) at (\x,\y);

\foreach \a/\b/\c in {1/2/4, 2/4/5, 2/3/5, 3/5/6, 5/6/7, 6/7/8}
	\filldraw [face] (v\a) -- (v\b) -- (v\c) -- cycle;

\foreach \i in {1,...,8}
	\node [vertex] at (v\i) {};

\node [rectangle, name=mybox, fit=(v5) (v6) (v7) (v8), inner sep=6pt, draw, thick, dotted, pin=left:$Y$] {};

\node [rectangle, fit=(v1) (mybox), inner sep=6pt, draw, thick, dotted, pin=left:$X$] {};

\end{tikzpicture}
\caption{$X$ and $Y$}
\end{subfigure}
\begin{subfigure}{0.3\textwidth}
\centering
\begin{tikzpicture}[yscale=0.87, baseline=-1.4cm]

\foreach \x/\y [count=\i] in {-0.5/1, 0.5/1, 1.5/1, 0/0, 1/0, 2/0, 1.5/-1, 2.5/-1}
	\coordinate (v\i) at (\x,\y);

\fill [face] (v3) -- (v5) -- (v6) -- cycle;
\foreach \a/\b/\c in {1/2/4, 2/4/5, 2/3/5}
	\filldraw [face] (v\a) -- (v\b) -- (v\c) -- cycle;

\draw [edge] (v3) -- (v6);
\draw [missing edge] (v5) -- (v6);

\foreach \i in {1,...,4}
	\node [vertex] at (v\i) {};
\foreach \i in {5,...,6}
	\node [missing vertex] at (v\i) {};


\end{tikzpicture}
\caption{$\abs{X \setminus Y}$}
\end{subfigure}
\begin{subfigure}{0.3\textwidth}
\centering
\begin{tikzpicture}[yscale=0.87, baseline=-1.4cm]

\foreach \x/\y [count=\i] in {-0.5/1, 0.5/1, 1.5/1, 0/0, 1/0, 2/0, 1.5/-1, 2.5/-1}
	\coordinate (v\i) at (\x,\y);

\foreach \a/\b/\c in {1/2/4}
	\filldraw [face] (v\a) -- (v\b) -- (v\c) -- cycle;

\draw [edge] (v2) -- (v3);

\foreach \i in {1,...,4}
	\node [vertex] at (v\i) {};


\end{tikzpicture}
\caption{$\abs{X - Y}$}
\end{subfigure}
\caption{A comparison of the different types of deletion} \label{fig:deletion}
\end{figure}
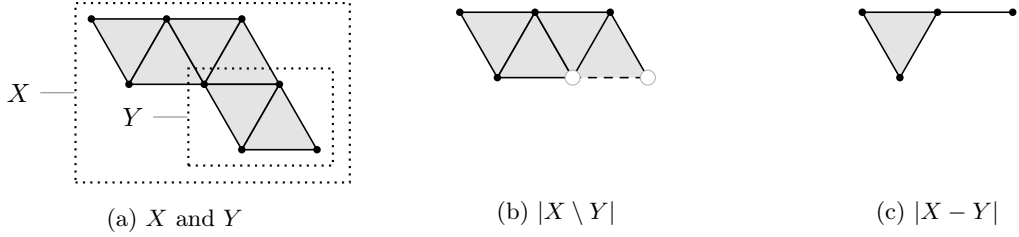

Given a face $\sigma$ of $X$, the \emph{(open) star} of $\sigma$ is the set of faces
\begin{equation*}
\st_X \sigma \coloneqq \{ \tau \in X : \sigma \subseteq \tau \}.
\end{equation*}
The geometric realisation of a star is always contractible, if $\sigma \neq \emptyset$. The \emph{link} of $\sigma$ is the set
\begin{align*}
\lk_X \sigma & \coloneqq \{ \tau \in X : \sigma \subset \tau, \ \tau \neq \sigma \} \\
& = \st_X \sigma \setminus \{\sigma\}.
\end{align*}
We will sometimes simply write ``$\st \sigma$'' and ``$\lk \sigma$'' if the space $X$ is clear from context. Note that if $\Delta$ is an abstract simplicial complex, the usual definition of a link is slightly different: we will refer to the usual definition as the ``simplicial link'', defined by
\begin{equation*}
\slk_\Delta \sigma \coloneq \{ \tau \in \Delta : \sigma \cup \tau \in \Delta, \ \sigma \cap \tau = \emptyset \}.
\end{equation*}
Although our link and the simplicial link are not the same, they are homotopy equivalent.

If $X$ is a polytopal complex, its \emph{barycentric subdivision} is the abstract simplicial complex $\bary(X)$ which has one vertex $v_\sigma$ for each non-empty face $\sigma$ of $X$, and a set $\{ v_{\sigma_1}, \dotsc, v_{\sigma_m} \}$ forms a face of $\bary(X)$ whenever $\{\sigma_1, \dotsc, \sigma_m\}$ is a chain in the poset of faces of $X$ ordered by inclusion (that is, $\sigma_1 \subset \dotsb \subset \sigma_m$, up to reordering). If each vertex $v_\sigma$ is positioned at the barycentre of the polytope $\sigma$, then the geometric realisation of $\bary(X)$ is exactly the geometric realisation of $X$, as a topological space. See \cref{fig:bary(X)} for an example.

\begin{figure}
\centering
\begin{tikzpicture}[scale=0.8]

\node (v1) [coordinate] at (0,0) {};
\node (v2) [coordinate] at (100:1.8) {};
\node (v3) [coordinate] at (-1.9,0.8) {};
\node (v4) [coordinate] at (-1.8,-0.9) {};
\node (v5) [coordinate] at (-0.1,-2.0) {};
\node (v6) [coordinate] at (10:1.8) {};
\node (v7) [coordinate] at ($(v6) + (100:1.8)$) {};

\filldraw [face] (v1)--(v2)--(v3)--(v4)--(v5)--cycle;
\filldraw [face] (v1)--(v5)--(v6)--cycle;
\filldraw [face] (v1)--(v2)--(v7)--(v6)--cycle;

\foreach \i/\j in {1/2, 2/3, 3/4, 4/5, 1/5, 1/6, 5/6, 2/7, 6/7}
	\node (v\i\j) [vertex] at ($(v\i)!0.5!(v\j)$) {};
\node (v12345) [vertex] at ($1/5*(v1)+1/5*(v2)+1/5*(v3)+1/5*(v4)+1/5*(v5)$) {};
\node (v156) [vertex] at ($1/3*(v1)+1/3*(v5)+1/3*(v6)$) {};
\node (v1267) [vertex] at ($1/4*(v1)+1/4*(v2)+1/4*(v6)+1/4*(v7)$) {};

\foreach \i in {1,...,5,12,23,34,45,15}
	\draw [edge, thin] (v\i)--(v12345);
\foreach \i in {1,5,6,15,16,56}
	\draw [edge, thin] (v\i)--(v156);
\foreach \i in {1,2,6,7,12,27,67,16}
	\draw [edge, thin] (v\i)--(v1267);

\foreach \i in {1,...,7}
	\node [vertex] at (v\i) {};

\end{tikzpicture}
\caption{The barycentric subdivision of the polytopal complex in \cref{fig:running-example}} \label{fig:bary(X)}
\end{figure}
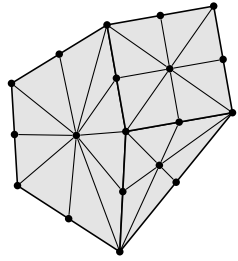

Given a polytopal complex $X$, the set of $k$-dimensional faces of $X$ will be denoted $X_k$. The \emph{$k$-skeleton} of $X$ is the subcomplex
\begin{equation*}
	\skel_k X \coloneqq \{ \sigma \in X : \dim \sigma \leq k \}.
\end{equation*}
The central definition in this paper is the \emph{$k$th co-skeleton} of $X$, which is the complement of the the $k$-skeleton:
\begin{align*}
\skel^c_k X & \coloneqq X \setminus \skel_k X \\
& = \{ \sigma \in X : \dim \sigma > k \}.
\end{align*}
We will mostly be interested in the topological properties of $\abs{\skel^c_k X} = \abs{X} \setminus \abs{\skel_k X}$. Note that the co-skeleton is not in general a polytopal complex, except in two special cases: $\skel^c_{-1} X$ is $X$ itself (modulo the empty face, which makes no difference to the topology), and if $X$ is $d$-dimensional, $\skel^c_d X$ is the polytopal complex with no faces.

\subsection{Topology}

We will use the notation $A \simeq B$ to indicate that the spaces $A$ and $B$ are homotopy equivalent.

We assume that the reader is familiar with homology and cohomology; for background, refer to \citet{book:Hatcher, book:Munkres-EAT}. While this paper does focus on cell complexes, many spaces we consider are not themselves cell complexes, so $H_i(A)$ will denote singular homology, with coefficients in $R$ where $R$ is a field or $\mathbb Z$. Reduced homology is denoted by $\tilde H_i(A)$. We use the convention that $\tilde H_{-1}(\emptyset) = R$. Analogous statements apply to cohomology, denoted $H^i(A)$.

For reference, here are two important theorems from algebraic topology that we will use repeatedly.

\begin{thm}[Mayer--Vietoris, {\citep[p.~149]{book:Hatcher}, \citep[Theorem~33.1]{book:Munkres-EAT}}] \label{thm:Mayer-Vietoris}
If $A$ and $B$ are open subsets of a topological space, then there is a long exact sequence:
\begin{equation*}
	\begin{tikzcd}[column sep=small]
		\dotsb \rar & H_i(A \cap B) \rar & H_i(A) \oplus H_i(B) \rar & H_i(A \cup B) \rar & H_{i-1}(A \cap B) \rar & \dotsb.
	\end{tikzcd}
\end{equation*}
If $A \cap B \neq \emptyset$, then we may replace these unreduced homology groups with reduced homology throughout.
\end{thm}

\begin{thm}[Nerve theorem, {\citep[Corollay~4G.3]{book:Hatcher}, \citep[Theorem~10.7]{art:Bjorner}}] \label{thm:nerve}
Suppose $\mathcal U = \{U_1, \dotsc, U_n\}$ is a family of open sets whose union is a paracompact space (e.g.\ any subspace of $\mathbb R^n$), or a family of closed sets whose union is a triangulable space. Suppose further that for every index set $I \subseteq \{1, \dotsc, n\}$, the intersection $\bigcap_{i \in I} U_i$ is either contractible or empty.

Construct a simplicial complex $N(\mathcal U)$ (called the ``nerve'' of $\mathcal U$) where the vertex set is $\{1, \dotsc, n\}$ and the set $I$ forms a face whenever $\bigcap_{i \in I} U_i$ is non-empty. Then
\begin{equation*}
\bigcup_{i=1}^n U_i \simeq \abs{N(\mathcal U)}.
\end{equation*}
\end{thm}

\section{First results about co-skeletons} \label{sec:general}

Let us begin to study the co-skeleton $\skel^c_k X$.

This first lemma tells us that while $\skel^c_k X$ is not itself a polytopal complex, it is homotopy equivalent to one. See \cref{fig:coskel-simeq-bary}.

\begin{lma} \label{thm:coskel-simeq-bary}
$\abs{\skel^c_k X} \simeq \abs{\bary(X) - \bary(\skel_k X)}$.
\end{lma}

\begin{figure}
\centering
\begin{subfigure}{0.4\textwidth}
\centering
\begin{tikzpicture}[scale=0.8]

\node (v1) [coordinate] at (0,0) {};
\node (v2) [coordinate] at (100:1.8) {};
\node (v3) [coordinate] at (-1.9,0.8) {};
\node (v4) [coordinate] at (-1.8,-0.9) {};
\node (v5) [coordinate] at (-0.1,-2.0) {};
\node (v6) [coordinate] at (10:1.8) {};
\node (v7) [coordinate] at ($(v6) + (100:1.8)$) {};

\foreach \i/\j in {1/2, 2/3, 3/4, 4/5, 1/5, 1/6, 5/6, 2/7, 6/7}
	\node (v\i\j) [coordinate] at ($(v\i)!0.5!(v\j)$) {};
\node (v12345) [coordinate] at ($1/5*(v1)+1/5*(v2)+1/5*(v3)+1/5*(v4)+1/5*(v5)$) {};
\node (v156) [coordinate] at ($1/3*(v1)+1/3*(v5)+1/3*(v6)$) {};
\node (v1267) [coordinate] at ($1/4*(v1)+1/4*(v2)+1/4*(v6)+1/4*(v7)$) {};

\filldraw [face] (v1)--(v2)--(v3)--(v4)--(v5)--cycle;
\filldraw [face] (v1)--(v5)--(v6)--cycle;
\filldraw [face] (v1)--(v2)--(v7)--(v6)--cycle;


\foreach \i in {1,...,7}
	\node [missing vertex] at (v\i) {};

\end{tikzpicture}
\caption{$\skel^c_0 X$}
\end{subfigure}
\begin{subfigure}{0.4\textwidth}
\centering
\begin{tikzpicture}[scale=0.8]

\node (v1) [coordinate] at (0,0) {};
\node (v2) [coordinate] at (100:1.8) {};
\node (v3) [coordinate] at (-1.9,0.8) {};
\node (v4) [coordinate] at (-1.8,-0.9) {};
\node (v5) [coordinate] at (-0.1,-2.0) {};
\node (v6) [coordinate] at (10:1.8) {};
\node (v7) [coordinate] at ($(v6) + (100:1.8)$) {};

\draw [sketch edge] (v2)--(v3)--(v4)--(v5)--(v6)--(v7)--cycle--(v1)--(v5) (v1)--(v6);

\foreach \i/\j in {1/2, 2/3, 3/4, 4/5, 1/5, 1/6, 5/6, 2/7, 6/7}
	\node (v\i\j) [vertex] at ($(v\i)!0.5!(v\j)$) {};
\node (v12345) [vertex] at ($1/5*(v1)+1/5*(v2)+1/5*(v3)+1/5*(v4)+1/5*(v5)$) {};
\node (v156) [vertex] at ($1/3*(v1)+1/3*(v5)+1/3*(v6)$) {};
\node (v1267) [vertex] at ($1/4*(v1)+1/4*(v2)+1/4*(v6)+1/4*(v7)$) {};

\foreach \ij in {12,23,34,45,15}
	\draw [edge] (v\ij)--(v12345);
\foreach \ij in {15,16,56}
	\draw [edge] (v\ij)--(v156);
\foreach \ij in {12,27,67,16}
	\draw [edge] (v\ij)--(v1267);


\end{tikzpicture}
\caption{$\bary(X) - \bary(\skel_0 X)$}
\end{subfigure}
\caption{\Cref{thm:coskel-simeq-bary} applied to \cref{fig:running-example}, with $k=0$} \label{fig:coskel-simeq-bary}
\end{figure}
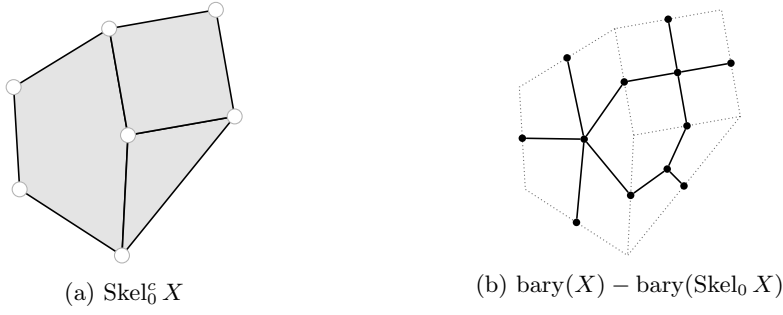

\begin{proof}
By definition,
\begin{align*}
	\abs{\skel^c_k X} & = \abs{X} \setminus \abs{\skel_k X} \\
	& = \abs{\bary(X)} \setminus \abs{\bary(\skel_k X)}.
\end{align*}
Now, $\bary(\skel_k X)$ is an induced subcomplex of the simplicial complex $\bary(X)$: if $v_{\sigma_1}, \dotsc, v_{\sigma_m}$ are vertices of $\bary(\skel_k X)$ that form a face of $\bary(X)$, then the faces $\sigma_1, \dotsc, \sigma_m$ form a chain in the face poset of $X$, so they still form a chain in the face poset of $\skel_k X$. Therefore, we may invoke \cref{thm:deletion-homotopy-equivalence}:
\begin{align*}
	\abs{\bary(X)} \setminus \abs{\bary(\skel_k X)} & \simeq \abs{\bary(X) - \bary(\skel_k X)}. \qedhere
\end{align*}
\end{proof}

\begin{crl} \label{thm:coskel-dimension}
$H_i(\skel^c_k \Delta) = 0$ when $i > d-k-1$.
\end{crl}

\begin{proof}
The vertices of $\bary(X) - \bary(\skel_k X)$ correspond to faces $\sigma$ of $X$ of dimensions between $k+1$ and $d$; therefore, the largest possible face of $\bary(X) - \bary(\skel_k X)$ has $d-k$ vertices, so its dimension is $d-k-1$.
\end{proof}

The name ``co-skeleton'' was chosen to suggest ``complement'', but also duality. Let us illustrate why.

\begin{prop} \label{thm:polytope-duality}
Suppose $P$ is a $(d+1)$-dimensional polytope. Then
\begin{equation*}
	\abs{\skel^c_k \partial P} \simeq \abs{\skel_{d-k-1} \partial P^*},
\end{equation*}
where $P^*$ is the polar dual polytope to $P$.
\end{prop}

\begin{proof}
By \cref{thm:coskel-simeq-bary},
\begin{equation*}
\abs{\skel^c_k \partial P} \simeq \abs{\bary(\partial P) - \bary(\skel_k \partial P)}.
\end{equation*}
The faces of $\bary(\partial P) - \bary(\skel_k \partial P)$ are the chains in the poset of faces of $\partial P$ of dimension greater than $k$. But the poset of non-empty faces of $\partial P$ is isomorphic to the poset of non-empty faces of $\partial P^*$, flipped upside down. Under this flip, the faces of $\partial P$ of dimension greater than $k$ become faces of $\partial P^*$ of dimension less than or equal to $d-k-1$. Therefore, $\bary(\partial P) - \bary(\skel_k \partial P) = \bary(\skel_{d-k-1} \partial P^*)$, so
\begin{align*}
	\abs{\bary(\partial P) - \bary(\skel_k \partial P)} & = \abs{\bary(\skel_{d-k-1} \partial P^*)} \\
	& = \abs{\skel_{d-k-1} \partial P^*}. \qedhere
\end{align*}
\end{proof}

Note that the fact that $P$ is a polytope is not essential to this proposition, just that $P$ has an associated ``dual cell structure''. In fact, if $X$ is a $d$-dimensional homology manifold, the space $\abs{\bary(X) - \bary(\skel_k X)}$ is exactly the $(d-k-1)$-skeleton of the dual cell structure used in some proofs of Poincar\'e duality --- for example, \citet[\S{}64]{book:Munkres-EAT} calls this space the ``dual $(d-k-1)$-skeleton'' of the manifold.

A related duality result is the Alexander duality theorem, which has direct implications for co-skeletons.
\begin{thm}[Alexander duality, {\citep[\S{}71, particularly Theorem~71.1 and Exercise~4]{book:Munkres-EAT}}] \label{thm:Alexander-duality}
Suppose $X$ is a polytopal complex where $\dim \Delta = d$ and $\abs{X}$ is homeomorphic to a sphere, and suppose $Y$ is a proper, non-empty subcomplex of $X$. Then
\begin{equation*}
\tilde H^i(\abs{Y}) \cong \tilde H_{d-i-1}(\abs{X} \setminus \abs{Y}) \qquad \text{and} \qquad \tilde H_i(\abs{Y}) \cong \tilde H^{d-i-1}(\abs{X} \setminus \abs{Y}).
\end{equation*}
\end{thm}

Taking $Y = \skel_k X$ gives us this corollary:

\begin{crl} \label{thm:Alexander-duality-for-coskeletons}
If $X$ is a $d$-dimensional polytopal complex where $\abs{X}$ is homeomorphic to a sphere, then for $k = 0, \dotsc, d-1$,
\begin{equation*}
\tilde H^i(\abs{\skel_k X}) \cong \tilde H_{d-i-1}(\abs{\skel^c_k X}) \qquad \text{and} \qquad \tilde H_i(\abs{\skel_k X}) \cong \tilde H^{d-i-1}(\abs{\skel^c_k X}).
\end{equation*}
\end{crl}

For a first example of an application for co-skeletons, let us consider neighbourly simplicial spheres. Suppose $\Delta$ is a simplicial sphere (that is, a simplicial complex where $\abs{\Delta}$ is homeomorphic to a sphere) with $n$ vertices. Then $\Delta$ is \emph{$t$-neighbourly} if every set of $t$ vertices is the vertex set of a face, or equivalently, if its $(t-1)$-skeleton is isomorphic to the $(t-1)$-skeleton of a simplex with $n$ vertices. For example, the boundaries of $d$-dimensional cyclic polytopes are $\lfloor d/2 \rfloor$-neighbourly. A major reason why neighbourly spheres have been studied is their connection with the Upper Bound Conjecture (see e.g.\ \citet{art:Alon-Kalai}).

The following lemma will let us characterise skeletons of simplices by their homology. (It follows without much difficulty from \citet[Theorem~3.1]{art:Bjorner-Kalai} and \citet[Equation~(3.6)]{art:Kalai-shifting}, but we also present a self-contained proof.)

\begin{lma} \label{thm:skel-simplex-maximises-homology}
Suppose $\Delta$ is a simplicial complex with $n$ vertices, $n \geq 1$, and $\dim \Delta = k < n$. Then
\begin{equation*}
\rank \tilde H_k(\abs{\Delta}; \mathbb Z) \leq \binom{n-1}{k+1},
\end{equation*}
with equality if and only if $\Delta$ is isomorphic to the $k$-skeleton of a simplex with $n$ vertices. If the coefficient ring is a field instead of $\mathbb Z$, the same holds with rank replaced by dimension.
\end{lma}

\begin{proof}
For simplicity of notation, assume the coefficient ring is $\mathbb Z$; the argument when the coefficient ring is a field is similar.

We will prove this by induction on $n$. For the base case of $n = 1$, there is only one simplicial complex with $1$ vertex: it is the $0$-skeleton of the simplex on $1$ vertex, and its homology matches the formula above with $k = 0$.

Now, for the inductive step, label the vertices of $\Delta$ as $\{1,\dotsc,n\}$, and define the following two open sets:
\begin{align*}
A & \coloneq \abs{\st n}, \\
B & \coloneq \abs{\Delta} \setminus \{n\}.
\end{align*}
Then $A \cup B$ is $\abs{\Delta}$, and $A \cap B$ is $\abs{\st n} \setminus \{n\} = \abs{\lk n}$. The Mayer--Vietoris theorem (\cref{thm:Mayer-Vietoris}) then gives us the following long exact sequence:
\begin{equation*}
\begin{tikzcd}
\dotsb \rar & \tilde H_i(\abs{\lk n}) \rar & \tilde H_i(\abs{\st n}) \oplus \tilde H_i(\abs{\Delta} \setminus \{n\}) \rar & \tilde H_i(\abs{\Delta}) \rar & \dotsb.
\end{tikzcd}
\end{equation*}
Since $\abs{\st n}$ is contractible, $\tilde H_i(\abs{\st n})$ is zero. Now, consider the following segment of the exact sequence:
\begin{equation*}
\begin{tikzcd}
\tilde H_k(\abs{\Delta} \setminus \{n\}) \rar & \tilde H_k(\abs{\Delta}) \rar & \tilde H_{k-1}(\abs{\lk n}).
\end{tikzcd}
\end{equation*}
By \cref{thm:deletion-homotopy-equivalence}, $\abs{\Delta} \setminus \{n\}$ is homotopy equivalent to $\abs{\Delta - n}$, which is a simplicial complex with $n-1$ vertices and dimension less than or equal to $k$, so by induction, $\rank \tilde H_k(\abs{\Delta} \setminus \{n\}) \leq \binom{n-2}{k+1}$. (Note that if the dimension of $\Delta - n$ is less than $k$, its degree $k$ homology must be zero.) Similarly, $\lk n$ is homotopy equivalent to the simplicial link of $n$, which is a simplicial complex of dimension at most $k-1$ and with at most $n-1$ vertices, so by induction, $\rank \tilde H_{k-1}(\abs{\lk n}) \leq \binom{n-2}{k}$. Therefore,
\begin{equation*}
\rank \tilde H_k(\abs{\Delta}) \leq \binom{n-2}{k+1} + \binom{n-2}{k} = \binom{n-1}{k+1}.
\end{equation*}
To obtain equality for $\rank \tilde H_k(\abs{\Delta})$, we need equality to hold for both $\rank \tilde H_k(\abs{\Delta - n})$ and $\rank \tilde H_{k-1}(\abs{\lk n})$, so by induction, both $\Delta - n$ and $\slk n$ need to be skeletons of simplices of the appropriate sizes. Therefore, $\Delta$ itself can only be a $k$-skeleton of a simplex on $n$ vertices.

Finally, note that the homology of a $j$-skeleton of a simplex is concentrated in degree $j$, since in all degrees less than $j$ the simplicial chain complex agrees with that of a full simplex, which is contractible. Therefore, if $\Delta$ is the $k$-skeleton of a simplex, the Mayer--Vietoris exact sequence is zero outside the segment above, so the desired equality does indeed hold.
\end{proof}

The previous lemma tells us that ``being a skeleton of a simplex'' is a homological property, if dimension and number of vertices are fixed. Combining this with \cref{thm:Alexander-duality-for-coskeletons}, we obtain the following characterisations of neighbourly spheres in terms of their skeletons and their co-skeletons:

\begin{crl} \label{thm:neighbourly-characterisation}
Suppose $\Delta$ is a $d$-dimensional simplicial sphere with $n$ vertices. Then for $t = 1, \dotsc, d$, the following are equivalent:
\begin{itemize}
\item $\Delta$ is $t$-neighbourly,
\item $\rank \tilde H_{t-1}(\abs{\skel_{t-1} \Delta}; \mathbb Z) = \binom{n-1}{t}$,
\item $\rank \tilde H^{d-t}(\abs{\skel^c_{t-1} \Delta}; \mathbb Z) = \binom{n-1}{t}$.
\end{itemize}
\end{crl}

\begin{proof}
By definition, $\Delta$ is $t$-neighbourly if and only if $\skel_{t-1} \Delta$ is isomorphic to the $(t-1)$-skeleton of a simplex on $n$ vertices. \Cref{thm:skel-simplex-maximises-homology} says that this occurs if and only if $\rank \tilde H_{t-1}(\abs{\skel_{t-1} \Delta}; \mathbb Z) = \binom{n-1}{t}$. And \cref{thm:Alexander-duality-for-coskeletons} says that
\begin{align*}
\tilde H_{t-1}(\abs{\skel_{t-1} \Delta}; \mathbb Z) & \cong \tilde H^{d-t}(\abs{\skel^c_{t-1} \Delta}; \mathbb Z). \qedhere
\end{align*}
\end{proof}

\section{The key long exact sequence} \label{sec:LES}

We now reach the main tool of this paper, which we will make much use of in the upcoming sections.

\begin{thm} \label{thm:LES}
If $X$ is a $d$-dimensional polytopal complex and $0 \leq k \leq d$, then there is the following long exact sequence:
\begin{equation*}
	\begin{tikzcd}[column sep=small]
		\dotsb \rar & \bigoplus_{\sigma \in X_k} \tilde H_i(\abs{\lk_X \sigma}) \rar & \tilde H_i(\abs{\skel^c_k X}) \rar & \tilde H_i(\abs{\skel^c_{k-1} X}) \rar & \bigoplus_{\sigma \in X_k} \tilde H_{i-1}(\abs{\lk_X \sigma}) \rar & \dotsb.
	\end{tikzcd}
\end{equation*}
\end{thm}

\begin{proof}
We will begin with the case $k = d$. The link of any $d$-dimensional face is empty, as is $\skel^c_d X$. The $(d-1)$th co-skeleton $\skel^c_{d-1} X$ consists only of the interiors of the $d$-dimensional faces, so it retracts onto $f_d(X)$ points. Therefore, the non-zero part of the claimed long exact sequence looks like this:
\begin{equation*}
\begin{tikzcd}
0 \rar & \tilde H_0(\skel^c_{d-1}) \arrow[d, phantom, "\cong" rotate=-90] \rar & \bigoplus_{\sigma \in X_d} \tilde H_{-1}(\lk \sigma) \arrow[d, phantom, "\cong" rotate=-90] \rar & \tilde H_{-1}(\skel^c_d X) \arrow[d, phantom, "\cong" rotate=-90] \rar & 0. \\
& R^{f_d(X) - 1} & R^{f_d(X)} & R
\end{tikzcd}
\end{equation*}
These homology groups do indeed form an exact sequence, with appropriate maps. (We won't need to know what the maps are in the rest of this paper, just that they exist.)

For the remainder of this proof, we will assume that $k < d$. We will use the Mayer--Vietoris theorem. Define the sets $A$ and $B$ by:
\begin{align*}
A & \coloneqq \bigcup_{\sigma \in X_k} \abs{\st_{\bary(X)} v_\sigma}, \\
B & \coloneqq \abs{\skel^c_k X} = \abs{X} \setminus \abs{\skel_k X}.
\end{align*}
See \cref{fig:Mayer-Vietoris-k=0,fig:Mayer-Vietoris-k=1} for an example. Both $A$ and $B$ are open in $\abs{X}$: $A$ is a union of open sets, and $B$ is the complement of a closed set.

\begin{figure}
\centering
\begin{subfigure}{0.35\textwidth}
\centering
\begin{tikzpicture}[scale=0.8] 

\node (v1) [coordinate] at (0,0) {};
\node (v2) [coordinate] at (100:1.8) {};
\node (v3) [coordinate] at (-1.9,0.8) {};
\node (v4) [coordinate] at (-1.8,-0.9) {};
\node (v5) [coordinate] at (-0.1,-2.0) {};
\node (v6) [coordinate] at (10:1.8) {};
\node (v7) [coordinate] at ($(v6) + (100:1.8)$) {};

\foreach \i/\j in {1/2, 2/3, 3/4, 4/5, 1/5, 1/6, 5/6, 2/7, 6/7}
	{
	\node (v\i\j) [coordinate] at ($(v\i)!0.5!(v\j)$) {};
	\node (v\j\i) [coordinate] at ($(v\i)!0.5!(v\j)$) {};
	}
\node (v12345) [coordinate] at ($1/5*(v1)+1/5*(v2)+1/5*(v3)+1/5*(v4)+1/5*(v5)$) {};
\node (v156) [coordinate] at ($1/3*(v1)+1/3*(v5)+1/3*(v6)$) {};
\node (v1267) [coordinate] at ($1/4*(v1)+1/4*(v2)+1/4*(v6)+1/4*(v7)$) {};


\foreach \i/\ix/\iy/\f in {1/12/15/12345, 2/12/23/12345, 3/23/34/12345, 4/34/45/12345, 5/45/15/12345, 1/15/16/156, 5/15/56/156, 6/16/56/156, 1/12/16/1267, 2/12/27/1267, 7/27/67/1267, 6/16/67/1267}
	{
	\coordinate (v\ix-\i) at ($(v\ix)!0.2!(v\i)$);
	\coordinate (v\iy-\i) at ($(v\iy)!0.2!(v\i)$);
	\coordinate (v\f-\i) at ($(v\f)!0.2!(v\i)$);
	\fill [face] (v\ix-\i) -- (v\i) -- (v\iy-\i) -- (v\f-\i) -- cycle;
	\draw [edge, scale around={0.8:(v\i)}] (v\ix-\i) -- (v\i) -- (v\iy-\i);
	\draw [missing edge] (v\ix-\i) -- (v\f-\i) -- (v\iy-\i);
	}

\foreach \i in {1,...,7}
	\node [vertex] at (v\i) {};

\end{tikzpicture}
\caption{$A = \bigsqcup_{\sigma \in X_0} \abs{\st_{\bary(X)} v_\sigma}$}
\end{subfigure}
\begin{subfigure}{0.35\textwidth}
\centering
\begin{tikzpicture}[scale=0.8] 

\node (v1) [coordinate] at (0,0) {};
\node (v2) [coordinate] at (100:1.8) {};
\node (v3) [coordinate] at (-1.9,0.8) {};
\node (v4) [coordinate] at (-1.8,-0.9) {};
\node (v5) [coordinate] at (-0.1,-2.0) {};
\node (v6) [coordinate] at (10:1.8) {};
\node (v7) [coordinate] at ($(v6) + (100:1.8)$) {};

\foreach \i/\j in {1/2, 2/3, 3/4, 4/5, 1/5, 1/6, 5/6, 2/7, 6/7}
	\node (v\i\j) [coordinate] at ($(v\i)!0.5!(v\j)$) {};
\node (v12345) [coordinate] at ($1/5*(v1)+1/5*(v2)+1/5*(v3)+1/5*(v4)+1/5*(v5)$) {};
\node (v156) [coordinate] at ($1/3*(v1)+1/3*(v5)+1/3*(v6)$) {};
\node (v1267) [coordinate] at ($1/4*(v1)+1/4*(v2)+1/4*(v6)+1/4*(v7)$) {};

\filldraw [face] (v1)--(v2)--(v3)--(v4)--(v5)--cycle;
\filldraw [face] (v1)--(v5)--(v6)--cycle;
\filldraw [face] (v1)--(v2)--(v7)--(v6)--cycle;


\foreach \i in {1,...,7}
	\node [missing vertex] at (v\i) {};

\end{tikzpicture}
\caption{$B = \abs{\skel^c_0 X}$}
\end{subfigure}
\par\medskip
\begin{subfigure}{0.35\textwidth}
\centering
\begin{tikzpicture}[scale=0.8] 

\node (v1) [coordinate] at (0,0) {};
\node (v2) [coordinate] at (100:1.8) {};
\node (v3) [coordinate] at (-1.9,0.8) {};
\node (v4) [coordinate] at (-1.8,-0.9) {};
\node (v5) [coordinate] at (-0.1,-2.0) {};
\node (v6) [coordinate] at (10:1.8) {};
\node (v7) [coordinate] at ($(v6) + (100:1.8)$) {};

\foreach \i/\j in {1/2, 2/3, 3/4, 4/5, 1/5, 1/6, 5/6, 2/7, 6/7}
	{
	\node (v\i\j) [coordinate] at ($(v\i)!0.5!(v\j)$) {};
	\node (v\j\i) [coordinate] at ($(v\i)!0.5!(v\j)$) {};
	}
\node (v12345) [coordinate] at ($1/5*(v1)+1/5*(v2)+1/5*(v3)+1/5*(v4)+1/5*(v5)$) {};
\node (v156) [coordinate] at ($1/3*(v1)+1/3*(v5)+1/3*(v6)$) {};
\node (v1267) [coordinate] at ($1/4*(v1)+1/4*(v2)+1/4*(v6)+1/4*(v7)$) {};


\foreach \i/\ix/\iy/\f in {1/12/15/12345, 2/12/23/12345, 3/23/34/12345, 4/34/45/12345, 5/45/15/12345, 1/15/16/156, 5/15/56/156, 6/16/56/156, 1/12/16/1267, 2/12/27/1267, 7/27/67/1267, 6/16/67/1267}
	{
	\coordinate (v\ix-\i) at ($(v\ix)!0.2!(v\i)$);
	\coordinate (v\iy-\i) at ($(v\iy)!0.2!(v\i)$);
	\coordinate (v\f-\i) at ($(v\f)!0.2!(v\i)$);
	\fill [face] (v\ix-\i) -- (v\i) -- (v\iy-\i) -- (v\f-\i) -- cycle;
	\draw [edge, scale around={0.8:(v\i)}] (v\ix-\i) -- (v\i) -- (v\iy-\i);
	\draw [missing edge] (v\ix-\i) -- (v\f-\i) -- (v\iy-\i);
	}

\foreach \i in {1,...,7}
	\node [missing vertex] at (v\i) {};

\end{tikzpicture}
\caption{$A \cap B = \bigsqcup_{\sigma \in X_0} \abs{\st_{\bary(X)} v_\sigma} \setminus \abs{\sigma}$}
\end{subfigure}
\begin{subfigure}{0.35\textwidth}
\centering
\begin{tikzpicture}[scale=0.8] 

\node (v1) [coordinate] at (0,0) {};
\node (v2) [coordinate] at (100:1.8) {};
\node (v3) [coordinate] at (-1.9,0.8) {};
\node (v4) [coordinate] at (-1.8,-0.9) {};
\node (v5) [coordinate] at (-0.1,-2.0) {};
\node (v6) [coordinate] at (10:1.8) {};
\node (v7) [coordinate] at ($(v6) + (100:1.8)$) {};

\foreach \i/\j in {1/2, 2/3, 3/4, 4/5, 1/5, 1/6, 5/6, 2/7, 6/7}
	\node (v\i\j) [coordinate] at ($(v\i)!0.5!(v\j)$) {};
\node (v12345) [coordinate] at ($1/5*(v1)+1/5*(v2)+1/5*(v3)+1/5*(v4)+1/5*(v5)$) {};
\node (v156) [coordinate] at ($1/3*(v1)+1/3*(v5)+1/3*(v6)$) {};
\node (v1267) [coordinate] at ($1/4*(v1)+1/4*(v2)+1/4*(v6)+1/4*(v7)$) {};

\filldraw [face] (v1)--(v2)--(v3)--(v4)--(v5)--cycle;
\filldraw [face] (v1)--(v5)--(v6)--cycle;
\filldraw [face] (v1)--(v2)--(v7)--(v6)--cycle;


\foreach \i in {1,...,7}
	\node [vertex] at (v\i) {};

\end{tikzpicture}
\caption{$A \cup B = \abs{\skel^c_{-1} X} = \abs{X}$}
\end{subfigure}
\caption{The sets $A$, $B$, $A \cap B$ and $A \cup B$ in the proof of \cref{thm:LES}, when $k = 0$ and $X$ is the polytopal complex in \cref{fig:running-example}} \label{fig:Mayer-Vietoris-k=0}
\end{figure}

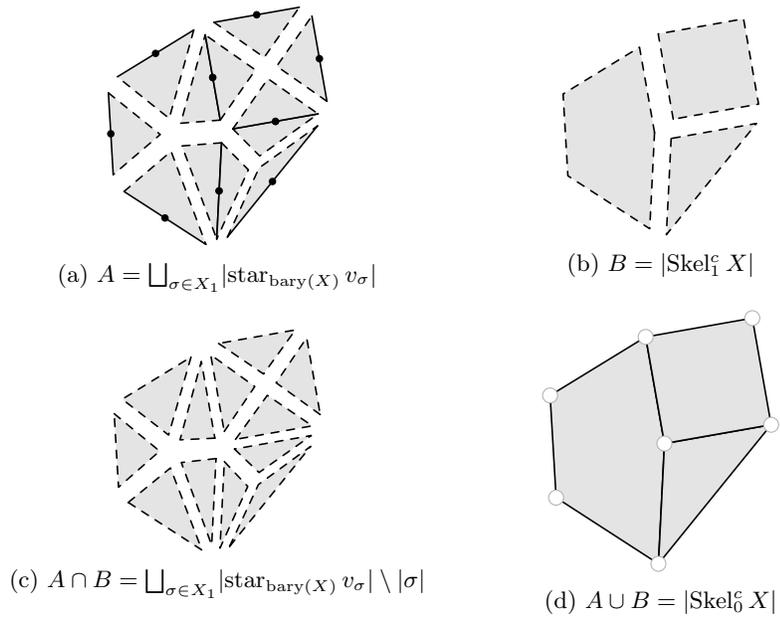
\begin{figure}
\centering
\begin{subfigure}{0.35\textwidth}
\centering
\begin{tikzpicture}[scale=0.8] 

\node (v1) [coordinate] at (0,0) {};
\node (v2) [coordinate] at (100:1.8) {};
\node (v3) [coordinate] at (-1.9,0.8) {};
\node (v4) [coordinate] at (-1.8,-0.9) {};
\node (v5) [coordinate] at (-0.1,-2.0) {};
\node (v6) [coordinate] at (10:1.8) {};
\node (v7) [coordinate] at ($(v6) + (100:1.8)$) {};

\foreach \i/\j in {1/2, 2/3, 3/4, 4/5, 1/5, 1/6, 5/6, 2/7, 6/7}
	{
	\node (v\i\j) [coordinate] at ($(v\i)!0.5!(v\j)$) {};
	}
\node (v12345) [coordinate] at ($1/5*(v1)+1/5*(v2)+1/5*(v3)+1/5*(v4)+1/5*(v5)$) {};
\node (v156) [coordinate] at ($1/3*(v1)+1/3*(v5)+1/3*(v6)$) {};
\node (v1267) [coordinate] at ($1/4*(v1)+1/4*(v2)+1/4*(v6)+1/4*(v7)$) {};


\foreach \x/\y/\f in {1/2/12345, 2/3/12345, 3/4/12345, 4/5/12345, 1/5/12345, 1/5/156, 1/6/156, 5/6/156, 1/2/1267, 2/7/1267, 6/7/1267, 1/6/1267}
	{
	\node (v\x-\x\y) [coordinate] at ($(v\x)!0.2!(v\x\y)$) {};
	\node (v\y-\x\y) [coordinate] at ($(v\y)!0.2!(v\x\y)$) {};
	\node (v\f-\x\y) [coordinate] at ($(v\f)!0.2!(v\x\y)$) {};
	\fill [face] (v\x-\x\y) -- (v\x\y) -- (v\y-\x\y) -- (v\f-\x\y) -- cycle;
	\draw [edge] (v\x-\x\y) -- (v\x\y) node [vertex] {} -- (v\y-\x\y);
	\draw [missing edge] (v\x-\x\y) -- (v\f-\x\y) -- (v\y-\x\y);
	}


\end{tikzpicture}
\caption{$A = \bigsqcup_{\sigma \in X_1} \abs{\st_{\bary(X)} v_\sigma}$}
\end{subfigure}
\begin{subfigure}{0.35\textwidth}
\centering
\begin{tikzpicture}[scale=0.8] 

\node (v1) [coordinate] at (0,0) {};
\node (v2) [coordinate] at (100:1.8) {};
\node (v3) [coordinate] at (-1.9,0.8) {};
\node (v4) [coordinate] at (-1.8,-0.9) {};
\node (v5) [coordinate] at (-0.1,-2.0) {};
\node (v6) [coordinate] at (10:1.8) {};
\node (v7) [coordinate] at ($(v6) + (100:1.8)$) {};

\foreach \i/\j in {1/2, 2/3, 3/4, 4/5, 1/5, 1/6, 5/6, 2/7, 6/7}
	\node (v\i\j) [coordinate] at ($(v\i)!0.5!(v\j)$) {};
\node (v12345) [coordinate] at ($1/5*(v1)+1/5*(v2)+1/5*(v3)+1/5*(v4)+1/5*(v5)$) {};
\node (v156) [coordinate] at ($1/3*(v1)+1/3*(v5)+1/3*(v6)$) {};
\node (v1267) [coordinate] at ($1/4*(v1)+1/4*(v2)+1/4*(v6)+1/4*(v7)$) {};


\foreach \i/\f in {1/12345, 2/12345, 3/12345, 4/12345, 5/12345, 1/156, 5/156, 6/156, 1/1267, 2/1267, 6/1267, 7/1267}
	\coordinate (v\i-\f) at ($(v\i)!0.2!(v\f)$);

\filldraw [face, missing edge] (v1-12345)--(v2-12345)--(v3-12345)--(v4-12345)--(v5-12345)--cycle;
\filldraw [face, missing edge] (v1-156)--(v5-156)--(v6-156)--cycle;
\filldraw [face, missing edge] (v1-1267)--(v2-1267)--(v7-1267)--(v6-1267)--cycle;



\end{tikzpicture}
\caption{$B = \abs{\skel^c_1 X}$}
\end{subfigure}
\par\medskip
\begin{subfigure}{0.35\textwidth}
\centering
\begin{tikzpicture}[scale=0.8] 

\node (v1) [coordinate] at (0,0) {};
\node (v2) [coordinate] at (100:1.8) {};
\node (v3) [coordinate] at (-1.9,0.8) {};
\node (v4) [coordinate] at (-1.8,-0.9) {};
\node (v5) [coordinate] at (-0.1,-2.0) {};
\node (v6) [coordinate] at (10:1.8) {};
\node (v7) [coordinate] at ($(v6) + (100:1.8)$) {};

\foreach \i/\j in {1/2, 2/3, 3/4, 4/5, 1/5, 1/6, 5/6, 2/7, 6/7}
	{
	\node (v\i\j) [coordinate] at ($(v\i)!0.5!(v\j)$) {};
	}
\node (v12345) [coordinate] at ($1/5*(v1)+1/5*(v2)+1/5*(v3)+1/5*(v4)+1/5*(v5)$) {};
\node (v156) [coordinate] at ($1/3*(v1)+1/3*(v5)+1/3*(v6)$) {};
\node (v1267) [coordinate] at ($1/4*(v1)+1/4*(v2)+1/4*(v6)+1/4*(v7)$) {};


\foreach \x/\y/\f in {1/2/12345, 2/3/12345, 3/4/12345, 4/5/12345, 1/5/12345, 1/5/156, 1/6/156, 5/6/156, 1/2/1267, 2/7/1267, 6/7/1267, 1/6/1267}
	{
	\node (v\x-\x\y-\f) [coordinate] at ($(v\x)!0.2!(v\x\y)!0.1!(v\f)$) {};
	\node (v\y-\x\y-\f) [coordinate] at ($(v\y)!0.2!(v\x\y)!0.1!(v\f)$) {};
	\node (v\f-\x\y) [coordinate] at ($(v\f)!0.2!(v\x\y)$) {};
	\fill [face] (v\x-\x\y-\f) -- (v\y-\x\y-\f) -- (v\f-\x\y) -- cycle;
	\draw [missing edge] (v\x-\x\y-\f) -- (v\f-\x\y) -- (v\y-\x\y-\f) -- cycle;
	}


\end{tikzpicture}
\caption{$A \cap B = \bigsqcup_{\sigma \in X_1} \abs{\st_{\bary(X)} v_\sigma} \setminus \abs{\sigma}$}
\end{subfigure}
\begin{subfigure}{0.35\textwidth}
\centering
\begin{tikzpicture}[scale=0.8] 

\node (v1) [coordinate] at (0,0) {};
\node (v2) [coordinate] at (100:1.8) {};
\node (v3) [coordinate] at (-1.9,0.8) {};
\node (v4) [coordinate] at (-1.8,-0.9) {};
\node (v5) [coordinate] at (-0.1,-2.0) {};
\node (v6) [coordinate] at (10:1.8) {};
\node (v7) [coordinate] at ($(v6) + (100:1.8)$) {};

\foreach \i/\j in {1/2, 2/3, 3/4, 4/5, 1/5, 1/6, 5/6, 2/7, 6/7}
	\node (v\i\j) [coordinate] at ($(v\i)!0.5!(v\j)$) {};
\node (v12345) [coordinate] at ($1/5*(v1)+1/5*(v2)+1/5*(v3)+1/5*(v4)+1/5*(v5)$) {};
\node (v156) [coordinate] at ($1/3*(v1)+1/3*(v5)+1/3*(v6)$) {};
\node (v1267) [coordinate] at ($1/4*(v1)+1/4*(v2)+1/4*(v6)+1/4*(v7)$) {};

\filldraw [face] (v1)--(v2)--(v3)--(v4)--(v5)--cycle;
\filldraw [face] (v1)--(v5)--(v6)--cycle;
\filldraw [face] (v1)--(v2)--(v7)--(v6)--cycle;


\foreach \i in {1,...,7}
	\node [missing vertex] at (v\i) {};

\end{tikzpicture}
\caption{$A \cup B = \abs{\skel^c_0 X}$}
\end{subfigure}
\caption{The sets $A$, $B$, $A \cap B$ and $A \cup B$ in the proof of \cref{thm:LES}, when $k = 1$ and $X$ is the polytopal complex in \cref{fig:running-example}} \label{fig:Mayer-Vietoris-k=1}
\end{figure}

First, we claim that the union in $A$ is in fact a \emph{disjoint} union. If two stars, say $\st_{\bary(X)} v_\sigma$ and $\st_{\bary(X)} v_\tau$, were to intersect, they would have a face of $\bary(X)$ in common, and by the definition of an open star, this face would need to contain both $v_\sigma$ and $v_\tau$. But faces in $\bary(X)$ correspond to chains in the face poset of $X$; therefore, no face in $\bary(X)$ can contain both $v_\sigma$ and $v_\tau$, since $\sigma$ and $\tau$ are both $k$-dimensional and are thus incomparable in the face poset. So $A$ is a disjoint union of stars. Since these stars are open, each one thus forms a connected component of $A$.

Next, let us consider $A \cap B$. Since $\st_{\bary(X)} v_\sigma$ does not intersect any faces of $X$ of dimension less than $k$, and the only $k$-dimensional face it intersects is $\sigma$, we conclude that
\begin{equation*}
	A \cap B = \bigsqcup_{\sigma \in X_k} \abs{\st_{\bary(X)} v_\sigma} \setminus \abs{\sigma}.
\end{equation*}
Now for $A \cup B$. Suppose $\tau$ is an arbitrary face of $X$. If $\dim \tau \leq k-1$, then $\abs \tau$ is disjoint from both $A$ and $B$; if $\dim \tau = k$, then $\abs \tau$ is a subset of $\abs{\st_{\bary(X)} v_\tau}$ hence $\abs \tau \subseteq A$, and if $\dim \tau \geq k+1$ then $\abs \tau$ is a subset of $B$. Therefore, $A \cup B$ consists of the relative interiors of all faces of dimension $k$ and greater; that is,
\begin{equation*}
	A \cup B = \abs{\skel^c_{k-1} X}.
\end{equation*}

We can now apply the Mayer--Vietoris theorem, with unreduced homology. This gives us the following long exact sequence:
{\small\begin{equation*}
	\begin{tikzcd}
		\dotsb \rar & H_i \bigg( \bigsqcup_{\sigma \in X_k} \abs{\st_{\bary(X)} v_\sigma} \setminus \abs{\sigma} \bigg) \rar & H_i \bigg( \bigsqcup_{\sigma \in X_k} \abs{\st_{\bary(X)} v_\sigma} \bigg) \oplus H_i(\abs{\skel^c_k X}) \rar & H_i(\abs{\skel^c_{k-1} X}) \rar & \dotsb.
	\end{tikzcd}
\end{equation*}}
Since unreduced homology of a disconnected space can be decomposed as a direct sum over components, we can simplify this long exact sequence:
{\small\begin{equation*}
	\begin{tikzcd}[column sep=small]
		\dotsb \rar & \bigoplus_{\sigma \in X_k} H_i \big( \abs{\st_{\bary(X)} v_\sigma} \setminus \abs{\sigma} \big) \rar & \bigg( \bigoplus_{\sigma \in X_k} H_i(\abs{\st_{\bary(X)} v_\sigma}) \bigg) \oplus H_i(\abs{\skel^c_k X}) \rar & H_i(\abs{\skel^c_{k-1} X}) \rar & \dotsb.
	\end{tikzcd}
\end{equation*}}

While this is a perfectly good long exact sequence, there are some steps we can do to ``clean it up'' and make it more convenient for later proofs. The first step is to convert from unreduced to reduced homology, which affects this sequence in degrees $0$ and $-1$. For each face $\sigma \in X_k$, there are two cases to consider: either $\sigma$ is not a facet, or it is.

If $\sigma$ is not a facet of $X$, then $\abs{\st_{\bary(X)} v_\sigma} \setminus \abs{\sigma}$ is non-empty, so $H_0(\abs{\st_{\bary(X)} v_\sigma} \setminus \abs{\sigma})$ is non-zero. The generators of a degree $0$ homology group are homotopy equivalence classes of maps from the $0$-simplex (i.e., a point) to the space in question; therefore, since the first map in our long exact sequence is induced by the inclusion of topological spaces, the map sends each generator of $H_0(\abs{\st_{\bary(X)} v_\sigma} \setminus \abs{\sigma})$ to a generator of $H_0(\abs{\st_{\bary(X)} v_\sigma}) \cong R$. Therefore, since homology groups in degree $0$ are always free, we may remove one summand of $R$ from each of these two homology groups. The result is that the unreduced homology group $H_0(\abs{\st_{\bary(X)} v_\sigma} \setminus \abs{\sigma})$ becomes reduced, and since $\abs{\st_{\bary(X)} v_\sigma}$ is contractible, its homology is zero outside degree $0$, so we can omit the $H_i(\abs{\st_{\bary(X)} v_\sigma})$ term. This holds for every $k$-face $\sigma$ that is not a facet.

What if $\sigma$ is a facet? In this case, $\abs{\st_{\bary(X)} v_\sigma} = \abs{\sigma}$, so $\abs{\st_{\bary(X)} v_\sigma} \setminus \abs{\sigma}$ is empty, hence its degree $0$ homology is zero in both the unreduced and the reduced form. However, its degree $(-1)$ reduced homology becomes non-zero. Therefore, we can do the following two actions:
\begin{itemize}
	\item First, the map from $H_0(\abs{\st_{\bary(X)} v_\sigma}) \cong R$ to $H_0(\abs{\skel^c_{k-1} X})$ takes generators to generators, so we can remove one summand of $R$ from each, by the same reasoning as above.
	\item Second, we can add a copy of the sequence $\dotsb \to 0 \to R \to R \to 0 \to \dotsb$ to our sequence, adding one summand of $R$ to each of $H_0(\abs{\skel^c_{k-1} X})$ and $H_{-1}(\abs{\st_{\bary(X)} v_\sigma} \setminus \abs{\sigma})$.
\end{itemize}
The result is the following:
\begin{itemize}
	\item The $H_i(\abs{\st_{\bary(X)} v_\sigma})$ term is removed (since it is zero when $i \neq 0$),
	\item No net change occurs to $H_0(\abs{\skel^c_{k-1} X})$, and
	\item The $H_{-1}(\abs{\st_{\bary(X)} v_\sigma} \setminus \abs{\sigma})$ term turns from $0$ to $R$, effectively converting from unreduced to reduced homology.
\end{itemize}

There is one more change to make in converting from unreduced to reduced homology: the map from $H_0(\abs{\skel^c_k X})$ to $H_0(\abs{\skel^c_{k-1} X})$ is induced by the inclusion of non-empty topological spaces (since $k < d$), so by a similar argument, we may replace both of these homology groups with reduced homology.

Our long exact sequence then becomes:
\begin{equation*}
	\begin{tikzcd}[column sep=small]
		\dotsb \rar & \bigoplus_{\sigma \in X_k} \tilde H_i \big( \abs{\st_{\bary(X)} v_\sigma} \setminus \abs{\sigma} \big) \rar & \tilde H_i(\abs{\skel^c_k X}) \rar & \tilde H_i(\abs{\skel^c_{k-1} X}) \rar & \dotsb.
	\end{tikzcd}
\end{equation*}

There is one last thing to do to this long exact sequence: simplifying $\tilde H_i \big( \abs{\st_{\bary(X)} v_\sigma} \setminus \abs{\sigma} \big)$. To do this, let us call on the Mayer--Vietoris theorem again. Assume that $\sigma$ is not a facet, so $\abs{\st_{\bary(X)} v_\sigma} \setminus \abs{\sigma}$ is non-empty. Define two new spaces:
\begin{align*}
A' & \coloneqq \abs{\st_{\bary(X)} v_\sigma}, \\
B' & \coloneqq \abs{\st_X \sigma} \setminus \abs{\sigma} = \abs{\lk_X \sigma}
\end{align*}

\begin{figure}
\centering
\begin{subfigure}{0.3\textwidth}
\centering
\begin{tikzpicture}[scale=0.8]

\node (v1) [coordinate] at (0,0) {};
\node (v2) [coordinate] at (100:1.8) {};
\node (v3) [coordinate] at (-1.9,0.8) {};
\node (v4) [coordinate] at (-1.8,-0.9) {};
\node (v5) [coordinate] at (-0.1,-2.0) {};
\node (v6) [coordinate] at (10:1.8) {};
\node (v7) [coordinate] at ($(v6) + (100:1.8)$) {};

\foreach \i/\j in {1/2, 2/3, 3/4, 4/5, 1/5, 1/6, 5/6, 2/7, 6/7}
	\node (v\i\j) [coordinate] at ($(v\i)!0.5!(v\j)$) {};
\node (v12345) [coordinate] at ($1/5*(v1)+1/5*(v2)+1/5*(v3)+1/5*(v4)+1/5*(v5)$) {};
\node (v156) [coordinate] at ($1/3*(v1)+1/3*(v5)+1/3*(v6)$) {};
\node (v1267) [coordinate] at ($1/4*(v1)+1/4*(v2)+1/4*(v6)+1/4*(v7)$) {};

\filldraw [face] (v1)--(v2)--(v3)--(v4)--(v5)--cycle;
\filldraw [face, fill=black!40] (v1)--(v5)--(v6)--cycle;
\filldraw [face] (v1)--(v2)--(v7)--(v6)--cycle;

\foreach \i in {1,...,7}
	\node [vertex] at (v\i) {};

\node [vertex, inner sep=1.5pt, pin=above left:$\sigma_1$] at (v3) {};

\draw [edge, line width=1.8pt] (v6)--(v7) node [coordinate, midway, pin=right:$\sigma_2$] {};

\node [coordinate, pin=below right:$\sigma_3$] at (v156) {};

\end{tikzpicture}
\caption{Three faces of $X$}
\end{subfigure}
\begin{subfigure}{0.6\textwidth}
\centering
\begin{tikzpicture}[scale=0.8]

\node (v1) [coordinate] at (0,0) {};
\node (v2) [coordinate] at (100:1.8) {};
\node (v3) [coordinate] at (-1.9,0.8) {};
\node (v4) [coordinate] at (-1.8,-0.9) {};
\node (v5) [coordinate] at (-0.1,-2.0) {};
\node (v6) [coordinate] at (10:1.8) {};
\node (v7) [coordinate] at ($(v6) + (100:1.8)$) {};

\foreach \i/\j in {1/2, 2/3, 3/4, 4/5, 1/5, 1/6, 5/6, 2/7, 6/7}
	\coordinate (v\i\j) at ($(v\i)!0.5!(v\j)$);
\coordinate (v12345) at ($1/5*(v1)+1/5*(v2)+1/5*(v3)+1/5*(v4)+1/5*(v5)$);
\coordinate (v156) at ($1/3*(v1)+1/3*(v5)+1/3*(v6)$);
\coordinate (v1267) at ($1/4*(v1)+1/4*(v2)+1/4*(v6)+1/4*(v7)$);

\fill [face] (v3)--(v23)--(v12345)--(v34)--cycle;
\fill [face] (v6)--(v1267)--(v7)--cycle;
\fill [face] (v1)--(v5)--(v6)--cycle;

\draw [sketch edge] (v1)--(v2)--(v3)--(v4)--(v5)--cycle;
\draw [sketch edge] (v5)--(v6)--(v1);
\draw [sketch edge] (v2)--(v7)--(v6);

\foreach \i in {1,...,5,12,23,34,45,15}
	\draw [sketch edge] (v\i)--(v12345);
\foreach \i in {1,5,6,15,16,56}
	\draw [sketch edge] (v\i)--(v156);
\foreach \i in {1,2,6,7,12,27,67,16}
	\draw [sketch edge] (v\i)--(v1267);

\draw [edge] (v23)--(v3)--(v34);
\draw [missing edge] (v23)--(v12345)--(v34);
\draw [edge] (v6)--(v7);
\draw [missing edge] (v6)--(v1267)--(v7);
\draw [missing edge] (v1)--(v5)--(v6)--cycle;

\node [vertex, pin=above left:$\abs{\st_{\bary(X)} v_{\sigma_1}}$] at (v3) {};
\node [vertex, pin=right:$\abs{\st_{\bary(X)} v_{\sigma_2}}$] at (v67) {};
\node [vertex, pin=below right:$\abs{\st_{\bary(X)} v_{\sigma_3}}$] at (v156) {};

\end{tikzpicture}
\caption{$\abs{\st_{\bary(X)} v_\sigma}$ for $\sigma = \sigma_1, \sigma_2, \sigma_3$}
\end{subfigure}
\par\medskip
\begin{subfigure}{0.6\textwidth}
\centering
\begin{tikzpicture}[scale=0.8]

\node (v1) [coordinate] at (0,0) {};
\node (v2) [coordinate] at (100:1.8) {};
\node (v3) [coordinate] at (-1.9,0.8) {};
\node (v4) [coordinate] at (-1.8,-0.9) {};
\node (v5) [coordinate] at (-0.1,-2.0) {};
\node (v6) [coordinate] at (10:1.8) {};
\node (v7) [coordinate] at ($(v6) + (100:1.8)$) {};

\foreach \i/\j in {1/2, 2/3, 3/4, 4/5, 1/5, 1/6, 5/6, 2/7, 6/7}
	\node (v\i\j) [coordinate] at ($(v\i)!0.5!(v\j)$) {};
\node (v12345) [coordinate] at ($1/5*(v1)+1/5*(v2)+1/5*(v3)+1/5*(v4)+1/5*(v5)$) {};
\node (v156) [coordinate] at ($1/3*(v1)+1/3*(v5)+1/3*(v6)$) {};
\node (v1267) [coordinate] at ($1/4*(v1)+1/4*(v2)+1/4*(v6)+1/4*(v7)$) {};

\foreach \i/\f in {1/12345, 2/12345, 3/12345, 4/12345, 5/12345, 12/12345, 23/12345, 34/12345, 45/12345, 15/12345, 1/156, 5/156, 6/156, 15/156, 16/156, 56/156, 1/1267, 2/1267, 6/1267, 7/1267, 12/1267, 27/1267, 67/1267, 16/1267}
	\coordinate (v\i-\f) at ($(v\i)!0.2!(v\f)$);



\fill [face] (v1-12345)--(v2-12345)--(v3-12345)--(v4-12345)--(v5-12345)--cycle;
\fill [face] (v1-156)--(v5-156)--(v6-156)--cycle;
\fill [face] (v1-1267)--(v2-1267)--(v7-1267)--(v6-1267)--cycle;

\draw [edge] (v2-12345)--(v3-12345)--(v4-12345);
\draw [missing edge] (v4-12345)--(v5-12345)--(v1-12345)--(v2-12345);
\draw [edge] (v6-1267)--(v7-1267);
\draw [missing edge] (v6-1267)--(v1-1267)--(v2-1267)--(v7-1267);
\draw [missing edge] (v1-156)--(v5-156)--(v6-156)--cycle;

\node [vertex, pin=above left:$\abs{\st_X \sigma_1}$] at (v3-12345) {};
\node [coordinate, pin=right:$\abs{\st_X \sigma_2}$] at (v67-1267) {};
\node [coordinate, pin=below right:$\abs{\st_X \sigma_3}$] at (v156) {};

\end{tikzpicture}
\caption{$\abs{\st_X \sigma}$ for $\sigma = \sigma_1, \sigma_2, \sigma_3$}	
\end{subfigure}

\caption{The spaces $\abs{\st_{\bary(X)} v_\sigma}$ and $\abs{\st_X \sigma}$ in the proof of \cref{thm:LES}, for various choices of $\sigma$} \label{fig:LES-stars}
\end{figure}
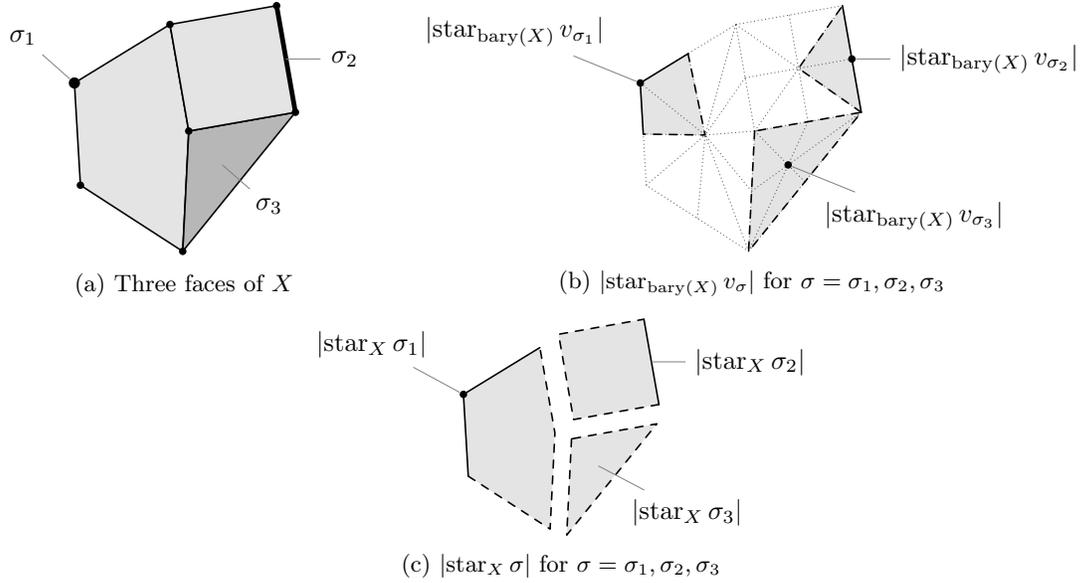

Since $\abs{\sigma} \subseteq \abs{\st_{\bary(X)} v_\sigma} \subseteq \abs{\st_X \sigma}$ --- for example, see \cref{fig:LES-stars} --- the intersection $A' \cap B'$ is $\abs{\st_{\bary(X)} v_\sigma} \setminus \abs{\sigma}$, and the union $A' \cup B'$ is simply $\abs{\st_X \sigma}$. Therefore, the Mayer--Vietoris long exact sequence with reduced homology for $A'$ and $B'$ is the following:
\begin{multline*}
	\begin{tikzcd}[ampersand replacement=\&]
		\dotsb \rar \& \tilde H_{i+1}(\abs{\st_X \sigma}) \rar \& \tilde H_i(\abs{\st_{\bary(X)} v_\sigma} \setminus \abs{\sigma})
	\end{tikzcd} \\
	\begin{tikzcd}[ampersand replacement=\&]
		{} \rar \& \tilde H_i(\abs{\st_{\bary(X)} v_\sigma}) \oplus \tilde H_i(\abs{\st_X \sigma} \setminus \abs{\sigma}) \rar \& \tilde H_i(\abs{\st_X \sigma}) \rar \& \dotsb.
	\end{tikzcd}
\end{multline*}
But stars of non-empty faces are contractible, so this becomes:
\begin{equation*}
	\begin{tikzcd}
		\dotsb \rar & 0 \rar & \tilde H_i(\abs{\st_{\bary(X)} v_\sigma} \setminus \abs{\sigma}) \rar & \tilde H_i(\abs{\st_X \sigma} \setminus \abs{\sigma}) \rar & 0 \rar & \dotsb.
	\end{tikzcd}
\end{equation*}
Therefore, $\tilde H_i(\abs{\st_{\bary(X)} v_\sigma} \setminus \abs{\sigma}) \cong \tilde H_i(\abs{\st_X \sigma} \setminus \abs{\sigma}) = \tilde H_i(\abs{\lk_X \sigma})$ for all non-facets $\sigma$. Finally, note that if $\sigma$ is a facet, this isomorphism still holds, since all spaces involved are empty.

Hence we obtain the desired long exact sequence.
\end{proof}

\section{Complexes characterised by link conditions} \label{sec:families}

Now that we have a long exact sequence relating co-skeletons and links, in this section we will consider some particular families of polytopal complexes that are characterised by links of faces having zero homology in some degrees.

\subsection{Cohen--Macaulay complexes}

A pure, $d$-dimensional polytopal complex $X$ is \emph{Cohen--Macaulay} if for every face $\sigma$, including $\sigma = \emptyset$, and every $i < d - \dim \sigma - 1$, the homology $\tilde H_i(\abs{\lk_\Delta \sigma})$ is $0$. If $\Delta \coloneqq X$ is a simplicial complex, this condition is equivalent to the Stanley--Reisner ring of $\Delta$ being a Cohen--Macaulay ring \citep[p.~1855]{art:Bjorner}.

\begin{rmk} \label{thm:CM-topological}
Cohen--Macaulayness is also a topological condition: $X$ is Cohen--Macaulay if and only if $\tilde H_i(\abs{X})$ and $\tilde H_i(\abs{X}, \abs{X} \setminus \{p\})$ are both $0$ for all $i < \dim X$ and all points $p \in \abs{X}$. This was proved by Munkres \citep[Corollary~3.4]{art:Munkres-TRC} for simplicial complexes, but a similar proof applies to polytopal complexes as well, with the following modified version of \citet[Lemma~3.3]{art:Munkres-TRC}:
\end{rmk}
\begin{lma}
If $X$ is a polytopal complex, $\sigma$ a face of $X$ and $p$ a point in the relative interior of $\sigma$, then $\tilde H_i(\abs{\lk_X \sigma}) \cong \tilde H_{i+\dim \sigma}(\abs{\lk_{\bary(X)} v_\sigma}) \cong \tilde H_{i + \dim \sigma + 1}(\abs{X}, \abs{X} \setminus \{p\})$.
\end{lma}

\begin{proof}
We saw in the proof of \cref{thm:LES} that
\begin{align*}
\tilde H_i(\abs{\lk_X \sigma}) & \cong \tilde H_i(\abs{\st_{\bary(X)} v_\sigma} \setminus \abs{\sigma}) \\
& = \tilde H_i \big( \abs{\lk_{\bary(X)} v_\sigma} \setminus (\abs{\sigma} \setminus \{v_\sigma\}) \big) \\
& \cong \tilde H_i(\abs{\slk_{\bary(X)} v_\sigma} \setminus \abs{\bary(\partial \sigma)}),
\end{align*}
where ``$\slk$'' denotes the simplicial link. Since $\bary(\partial \sigma)$ is an induced subcomplex of the simplicial complex $\slk_{\bary(X)} v_\sigma$, \cref{thm:deletion-homotopy-equivalence} says that
\begin{align*}
\abs{\slk_{\bary(X)} v_\sigma} \setminus \abs{\bary(\partial \sigma)} & \simeq \abs{\slk_{\bary(X)} v_\sigma - \bary(\partial \sigma)}.
\end{align*}
It is not hard to check that
\begin{align*}
\slk_{\bary(X)} v_\sigma & = \big( \slk_{\bary(X)} v_\sigma - \bary(\partial \sigma) \big) * \bary(\partial \sigma),
\end{align*}
where $*$ denotes simplicial join \citep[\S{}62]{book:Munkres-EAT}, so \citet[Theorem~62.5]{book:Munkres-EAT} implies that
\begin{align*}
\tilde H_i(\abs{\slk_{\bary(X)} v_\sigma - \bary(\partial \sigma)}) & \cong \tilde H_{i+\dim \sigma}(\abs{\slk_{\bary(X)} v_\sigma}).
\end{align*}
Finally, since $\bary(X)$ is a simplicial complex, we can apply \citet[Lemma~3.3]{art:Munkres-TRC} to conclude that
\begin{align*}
\tilde H_{i+\dim \sigma}(\abs{\slk_{\bary(X)} v_\sigma}) & \cong \tilde H_{i+\dim \sigma + 1}(\abs{X}, \abs{X} \setminus \{v_\sigma\}),
\end{align*}
and note that although $v_\sigma$ is typically defined to be the barycentre of $\sigma$, it can be chosen to be any point in the relative interior of $\sigma$.
\end{proof}

The following theorem allows us to determine whether a complex is Cohen--Macaulay by considering its co-skeletons.

\begin{thm} \label{thm:CM-characterisation}
A $d$-dimensional polytopal complex $X$ is Cohen--Macaulay if and only if it is pure and $\tilde H_i(\abs{\skel^c_k X}) = 0$ for all $i < d-k-1$ and all $k = -1, \dotsc, d$.
\end{thm}

Note that this theorem combined with \cref{thm:coskel-dimension} means that the homology of $\abs{\skel^c_k X}$ is entirely concentrated in degree $d-k-1$.

\begin{proof}[Proof of \cref{thm:CM-characterisation}]
To begin with, assume $X$ is Cohen--Macaulay. We will prove that $\tilde H_i(\abs{\skel^c_k X}) = 0$ for $i < d-k-1$ by induction on $k$.

For the base case, when $k = -1$, we have
\begin{equation*}
\tilde H_i(\abs{\skel^c_{-1} X}) = \tilde H_i(\abs{X}) = \tilde H_i(\abs{\lk_X \emptyset}),
\end{equation*}
which is zero for all $i < d$ since $X$ is Cohen--Macaulay.

Now, suppose $k > -1$, and consider this part of the long exact sequence from \cref{thm:LES}:
\begin{equation*}
\begin{tikzcd}
\bigoplus_{\sigma \in X_k} \tilde H_i(\abs{\lk \sigma}) \rar & \tilde H_i(\abs{\skel^c_k X}) \rar & \tilde H_i(\abs{\skel^c_{k-1} X}).
\end{tikzcd}
\end{equation*}
Since $X$ is Cohen--Macaulay and $i < d-k-1$, we know that $\tilde H_i(\abs{\lk \sigma})$ is $0$. Also, since $i$ is less than $d-k-1$, it is certainly less than $d-(k-1)-1$, so by induction, $\tilde H_i(\abs{\skel^c_{k-1} X})$ is $0$. Therefore, $\tilde H_i(\abs{\skel^c_k X})$ must be $0$ too, and we are done with this direction of the proof.

Now, for the reverse direction, assume that $\tilde H_i(\abs{\skel^c_k X}) = 0$ for all $i < d-k-1$ and all $k = -1, \dotsc, d$. Consider this part of the long exact sequence for $k \geq 0$:
\begin{equation*}
\begin{tikzcd}
\tilde H_{i+1}(\abs{\skel^c_{k-1} X}) \rar & \bigoplus_{\sigma \in X_k} \tilde H_i(\abs{\lk \sigma}) \rar & \tilde H_i(\abs{\skel^c_k X})
\end{tikzcd}
\end{equation*}
By our assumption, $\tilde H_i(\abs{\skel^c_k X}) = 0$. Also, since $i < d-k-1$, we have $i+1 < d - (k-1) - 1$, so our assumption tells us that $\tilde H_{i+1}(\abs{\skel^c_{k-1} X}) = 0$ too. Therefore, $\bigoplus_{\sigma \in X_k} \tilde H_i(\abs{\lk \sigma})$ is $0$, so $\tilde H_i(\abs{\lk \sigma})$ must be $0$ for all $i < d-k-1 = d- \dim \sigma - 1$.

This is true for all faces $\sigma$ of dimension $k \geq 0$, which leaves the case $\sigma = \emptyset$. In this case,
\begin{equation*}
\tilde H_i(\abs{\lk \sigma}) = \tilde H_i(\abs{X}) = \tilde H_i(\abs{\skel^c_{-1} X}) = 0.
\end{equation*}
Therefore, $X$ is Cohen--Macaulay.
\end{proof}

\begin{rmk}
This theorem has potential computational applications. Na\"\i{}vely, to check whether some complex $X$ is Cohen--Macaulay, one must compute the homology groups of the link of every face of $X$ --- if $\dim X = d$, there are at least $2^{d+1}$ faces. But with this result, one only needs to compute homologies of the $k$th co-skeletons of $X$ for the $d+2$ values of $k$ between $-1$ and $d$, or equivalently (by \cref{thm:coskel-simeq-bary}) the simplicial complexes $\bary(X) - \bary(\skel_k X)$. The tradeoff is that these simplicial complexes are much larger and more complicated than the links of faces: for instance, the number of vertices of $\bary(X) - \bary(\skel_k X)$ is at least on the order of $2^d$.
\end{rmk}

\begin{rmk}
We saw in \cref{thm:neighbourly-characterisation} that neighbourly spheres are characterised by the homology groups of either their skeletons or their co-skeletons. However, for Cohen--Macaulay complexes, the homology groups of skeletons are not sufficient for a characterisation: for example, \cref{fig:CM-not-characterised-by-skeletons} shows two complexes, one Cohen--Macaulay and one not, whose $k$-skeletons have isomorphic $i$th homology groups for all $k$ and $i$.
\end{rmk}

\begin{figure}
\centering
\begin{subfigure}{0.4\textwidth}
\centering
\begin{tikzpicture}

\coordinate (v0) at (0,0);
\foreach \i in {1,...,6}
	\coordinate (v\i) at (60*\i:1);

\filldraw [face] (v0)--(v1)--(v2)--cycle;
\filldraw [face] (v0)--(v2)--(v3)--cycle;
\filldraw [face] (v0)--(v3)--(v4)--cycle;
\filldraw [face] (v0)--(v4)--(v5)--cycle;
\filldraw [face] (v0)--(v5)--(v6)--cycle;
\filldraw [face] (v0)--(v6)--(v1)--cycle;

\node [vertex] at (v0){};
\foreach \i in {1,...,6}
	\node [vertex] at (v\i) {};

\end{tikzpicture}
\caption{A Cohen--Macaulay simplicial complex}
\end{subfigure} \qquad
\begin{subfigure}{0.4\textwidth}
\centering
\begin{tikzpicture}

\coordinate (v0) at (0,0);
\coordinate (v1) at (1,0);
\coordinate (v2) at ($(v1)+(60:1)$);
\coordinate (v3) at ($(v1)+(-60:1)$);
\coordinate (v4) at (-1,0);
\coordinate (v5) at ($(v4)+(120:1)$);
\coordinate (v6) at ($(v4)+(-120:1)$);

\filldraw [face] (v1)--(v0)--(v2)--cycle;
\filldraw [face] (v1)--(v2)--(v3)--cycle;
\filldraw [face] (v1)--(v3)--(v0)--cycle;
\filldraw [face] (v4)--(v0)--(v5)--cycle;
\filldraw [face] (v4)--(v5)--(v6)--cycle;
\filldraw [face] (v4)--(v6)--(v0)--cycle;

\node [vertex] at (v0){};
\foreach \i in {1,...,6}
	\node [vertex] at (v\i) {};

\end{tikzpicture}
\caption{A non-Cohen--Macaulay simplicial complex}
\end{subfigure}
\caption{Two simplicial complexes whose skeletons have isomorphic homology groups} \label{fig:CM-not-characterised-by-skeletons}
\end{figure}
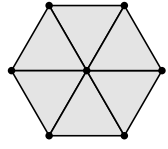
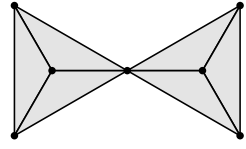

\Cref{thm:CM-characterisation} says that if $X$ is Cohen--Macaulay, almost all of the long exact sequence from \cref{thm:LES} is $0$. The remaining non-zero terms are captured in the following corollary:

\begin{crl} \label{thm:CM-SES}
If $X$ is a Cohen--Macaulay polytopal complex and $k = 0, \dotsc, d$, we have the following short exact sequence:
\begin{equation*}
\begin{tikzcd}
0 \rar & \tilde H_{d-k}(\abs{\skel^c_{k-1} X}) \rar & \bigoplus_{\sigma \in X_k} \tilde H_{d-k-1}(\abs{\lk \sigma}) \rar & \tilde H_{d-k-1}(\abs{\skel^c_k X}) \rar & 0.
\end{tikzcd}
\end{equation*}
\end{crl}

\begin{proof}
Take the following segment of the long exact sequence from \cref{thm:LES}:
{\small \begin{equation*}
\begin{tikzcd}
\tilde H_{d-k}(\abs{\skel^c_k X}) \rar & \tilde H_{d-k}(\abs{\skel^c_{k-1} X}) \rar & \bigoplus_{\sigma \in X_k} \tilde H_{d-k-1}(\abs{\lk \sigma}) \rar & \tilde H_{d-k-1}(\abs{\skel^c_k X}) \rar & \tilde H_{d-k-1}(\abs{\skel^c_{k-1} X}).
\end{tikzcd}
\end{equation*}}
The first term, $\tilde H_{d-k}(\abs{\skel^c_k X})$, is $0$ by \cref{thm:coskel-dimension}. The last term, $\tilde H_{d-k-1}(\abs{\skel^c_{k-1} X})$, is $0$ by \cref{thm:CM-characterisation}.
\end{proof}

We can combine these short exact sequences into a long exact sequence, using the following homological algebra fact:
\begin{lma} \label{thm:splicing}
Suppose we have two exact sequences:
\begin{equation*}
\begin{tikzcd}
\dotsb \rar & A_2 \arrow[r, "a_2"] & A_1 \arrow[r, "a_1"] & C \rar & 0
\end{tikzcd}
\qquad \text{and} \qquad
\begin{tikzcd}
0 \rar & C \arrow[r, "b_0"] & B_{-1} \arrow[r, "b_{-1}"] & B_{-2} \rar & \dotsb.
\end{tikzcd}
\end{equation*}
Then the horizontal sequence in the following diagram is exact:
\begin{equation*}
\begin{tikzpicture}[commutative diagrams/every diagram, column sep=1.0cm, row sep=1.0cm]

\matrix[matrix of math nodes, name=m, commutative diagrams/every cell] {
& & 0 & 0 \\
\dotsb & A_2 & A_1 & B_{-1} & B_{-2} & \dotsb. \\
};
\node (c) at ($(m-2-3)!0.5!(m-1-4)$) {$C$};

\path[commutative diagrams/.cd, every arrow, every label]
	(m-1-3) edge (c)
	(c) edge (m-1-4)
		edge node {$b_0$} (m-2-4)
	(m-2-1) edge (m-2-2)
	(m-2-2) edge node {$a_2$} (m-2-3)
	(m-2-3) edge node {$a_1$} (c)
		edge[commutative diagrams/dashed] node [swap] {$b_0 \circ a_1$} (m-2-4)
	(m-2-4) edge node {$b_{-1}$} (m-2-5)
	(m-2-5) edge (m-2-6);

\end{tikzpicture}
\end{equation*}
\end{lma}

\begin{proof}
Since $a_1$ is surjective, $\im(b_0 \circ a_1) = \im(b_0) = \ker(b_{-1})$. Since $b_0$ is injective, $\ker(b_0 \circ a_1) = \ker a_1 = \im a_2$.
\end{proof}

\begin{crl} \label{thm:CM-LES}
If $X$ is Cohen--Macaulay, we have the following long exact sequence:
\begin{equation*}
\begin{tikzcd}
0 \rar & \tilde H_d(\abs{X}) \rar & \bigoplus_{\sigma \in X_0} \tilde H_{d-1}(\abs{\lk \sigma}) \rar & \bigoplus_{\sigma \in X_1} \tilde H_{d-2}(\abs{\lk \sigma}) \rar & \dotsb \rar & \bigoplus_{\sigma \in X_d} \tilde H_{-1}(\abs{\lk \sigma}) \rar & R \rar & 0.
\end{tikzcd}
\end{equation*}
\end{crl}

\begin{proof}
Use \cref{thm:splicing} to stitch together the short exact sequences from \cref{thm:CM-SES}. The start of the sequence looks like this:
{\small \begin{equation*}
\begin{tikzpicture}[commutative diagrams/every diagram, column sep=0.6cm, row sep=1.6cm]

\matrix[matrix of math nodes, name=m, commutative diagrams/every cell] {
& & 0 & 0 \\
0 & \tilde H_d(\abs{\skel^c_{-1} X}) & \bigoplus_{\sigma \in X_0} \tilde H_{d-1}(\abs{\lk \sigma}) & \bigoplus_{\sigma \in X_1} \tilde H_{d-2}(\abs{\lk \sigma}) & \bigoplus_{\sigma \in X_2} \tilde H_{d-3}(\abs{\lk \sigma}) & \dotsb \\
& & & 0 & 0 \\
};
\node (hd1skel) at ($(m-2-3)!0.5!(m-1-4)$) {$\tilde H_{d-1}(\abs{\skel^c_0 X})$};
\node (hd2skel) at ($(m-3-4)!0.5!(m-2-5)$) {$\tilde H_{d-2}(\abs{\skel^c_1 X})$};

\path[commutative diagrams/.cd, every arrow, every label]
	(m-1-3) edge (hd1skel)
	(hd1skel) edge (m-1-4)
		edge (m-2-4)
	(m-2-1) edge (m-2-2)
	(m-2-2) edge (m-2-3)
	(m-2-3) edge (hd1skel)
		edge[commutative diagrams/dashed] (m-2-4)
	(m-2-4) edge[commutative diagrams/dashed] (m-2-5)
		edge (hd2skel)
	(m-2-5) edge[commutative diagrams/dashed] (m-2-6)
	(hd2skel) edge (m-2-5)
		edge (m-3-5)
	(m-3-4) edge (hd2skel);

\end{tikzpicture}
\end{equation*}}
And the end like this:
{\small \begin{equation*}
\begin{tikzpicture}[commutative diagrams/every diagram, column sep=0.6cm, row sep=1.6cm]

\matrix[matrix of math nodes, name=m, commutative diagrams/every cell] {
& 0 & 0 \\
\dotsb & \bigoplus_{\sigma \in X_{d-2}} \tilde H_1(\abs{\lk \sigma}) & \bigoplus_{\sigma \in X_{d-1}} \tilde H_0(\abs{\lk \sigma}) & \bigoplus_{\sigma \in X_d} \tilde H_{-1}(\abs{\lk \sigma}) & \tilde H_{-1}(\abs{\skel^c_d X}) & 0 \\
& & 0 & 0 \\
};
\node (h1skel) at ($(m-2-2)!0.5!(m-1-3)$) {$\tilde H_1(\abs{\skel^c_{d-2} X})$};
\node (h0skel) at ($(m-3-3)!0.5!(m-2-4)$) {$\tilde H_0(\abs{\skel^c_{d-1} X})$};

\path[commutative diagrams/.cd, every arrow, every label]
	(m-1-2) edge (h1skel)
	(h1skel) edge (m-1-3)
		edge (m-2-3)
	(m-2-1) edge[commutative diagrams/dashed] (m-2-2)
	(m-2-2) edge[commutative diagrams/dashed] (m-2-3)
		edge (h1skel)
	(m-2-3) edge[commutative diagrams/dashed] (m-2-4)
		edge (h0skel)
	(m-2-4) edge (m-2-5)
	(m-2-5) edge (m-2-6)
	(h0skel) edge (m-2-4)
		edge (m-3-4)
	(m-3-3) edge (h0skel);

\end{tikzpicture}
\end{equation*}}

\end{proof}

Note that this exact sequence is reminiscent of the ``partition complex'', a chain complex defined by \citet[Definition~25]{art:Adiprasito-Yashfe}.

\begin{rmk}
If the definition of a Cohen--Macaulay complex is modified to allow $\tilde H_i(\abs{\lk_X \sigma})$ to be non-zero when $\sigma = \emptyset$, we get the definition of a \emph{Buchsbaum} complex. Buchsbaum complexes can also be characterised by homological properties of their co-skeletons, although the statement is not as simple as for Cohen--Macaulay complexes in \cref{thm:CM-characterisation}: a complex $X$ is Buchsbaum if and only if the map $\tilde H_i(\abs{\skel^c_k X}) \to \tilde H_i(\abs{X})$ induced by the inclusion of topological spaces is an isomorphism for all $i < d-k-1$ and a surjection for $i = d-k-1$, for all $k = -1, \dotsc, d$. This follows fairly easily from the long exact sequence in \cref{thm:LES}, although we leave the details to the reader.
\end{rmk}

\subsection{Leray complexes}

A simplicial complex $\Delta$ is \emph{$r$-Leray} if every induced subcomplex $\Lambda$ of $\Delta$ has $\tilde H_i(\abs{\Lambda}) = 0$ for $i \geq r$. For example, if $\Delta$ is the nerve of a family of convex open subsets of $\mathbb R^r$, then it follows from the nerve theorem (\cref{thm:nerve}) that $\Delta$ is $r$-Leray (see \citet{art:Kalai-Meshulam}).

Equivalently, $\Delta$ is $r$-Leray if the homology $\tilde H_i(\abs{\lk_\Delta \sigma})$ is $0$ for every face $\sigma \in \Delta$ (including the empty face) and every $i \geq r$ \citep[Proposition~3.1]{art:Kalai-Meshulam}. We will use this condition to generalise $r$-Leray-ness to polytopal complexes:  we will say a complex $X$ is $r$-Leray if $\tilde H_i(\abs{\lk_X \sigma}) = 0$ for all faces $\sigma$ (including $\sigma = \emptyset$) and all $i \geq r$. Note that this condition is not equivalent to the condition on induced subcomplexes in the non-simplicial case: for example, if $X$ is a square, every non-empty link is contractible so $X$ is $0$-Leray, but the subcomplex induced by a pair of diagonally opposite vertices is not connected.

The following theorem provides a characterisation of $r$-Leray complexes in terms of their co-skeletons.

\begin{thm} \label{thm:Leray-characterisation}
A polytopal complex $X$ is $r$-Leray if and only if $\tilde H_i(\abs{\skel^c_k X}) = 0$ for all $i \geq r$ and all $k = -1, \dotsc, d$.
\end{thm}

\begin{proof}
For the forward direction, assume $X$ is $r$-Leray and let $i \geq r$. When $k = -1$,
\begin{equation*}
\tilde H_i(\abs{\skel^c_{-1} X}) = \tilde H_i(\abs{X}) = \tilde H_i(\abs{\lk_X \emptyset}),
\end{equation*}
which is $0$, since $X$ is $r$-Leray.

For $k \geq 0$, consider the following subsequence of the long exact sequence from \cref{thm:LES}:
\begin{equation*}
\begin{tikzcd}[column sep=small]
\bigoplus_{\sigma \in X_k} \tilde H_i(\abs{\lk \sigma}) \rar & \tilde H_i(\abs{\skel^c_k X}) \rar & \tilde H_i(\abs{\skel^c_{k-1} X}).
\end{tikzcd}
\end{equation*}
Since $X$ is $r$-Leray, $\bigoplus_{\sigma \in X_k} \tilde H_i(\abs{\lk \sigma})$ is $0$. Therefore, $\tilde H_i(\abs{\skel^c_k X})$ injects into $\tilde H_i(\abs{\skel^c_{k-1} X})$. Since this is true for all $k = 0, \dotsc, d$, we get a series of injections:
\begin{equation*}
\begin{tikzcd}[column sep=small]
\tilde H_i(\abs{\skel^c_k X}) \arrow[r, hook] & \tilde H_i(\abs{\skel^c_{k-1} X}) \arrow[r, hook] & \tilde H_i(\abs{\skel^c_{k-2} X}) \arrow[r, hook] & \dotsb \arrow[r, hook] & \tilde H_i(\abs{\skel^c_0 X}) \arrow[r, hook] & \tilde H_i(\abs{\skel^c_{-1} X}).
\end{tikzcd}
\end{equation*}
But we saw above that $\tilde H_i(\abs{\skel^c_{-1} X}) = 0$. Therefore, $\tilde H_i(\abs{\skel^c_k X})$ injects into $0$, so it must itself be $0$, for all $k = 0, \dotsc, d$.

For the reverse direction, assume that $\tilde H_i(\abs{\skel^c_k X}) = 0$ for all $k = -1, \dotsc, d$ and all $i \geq r$. When $k \geq 0$, consider this part of the long exact sequence:
\begin{equation*}
\begin{tikzcd}
\tilde H_{i+1}(\abs{\skel^c_{k-1} X}) \rar & \bigoplus_{\sigma \in X_k} \tilde H_i(\abs{\lk \sigma}) \rar & \tilde H_i(\abs{\skel^c_k X}).
\end{tikzcd}
\end{equation*}
By our assumptions, the two outer terms here are both $0$, hence the middle term is $0$ as well, so $\tilde H_i(\abs{\lk \sigma}) = 0$ for all $k$-faces $\sigma$. This works for all $k = 0, \dotsc, d$, which only leaves $k = -1$: the only $(-1)$-dimensional face is $\emptyset$, and
\begin{equation*}
\tilde H_i(\abs{\lk_X \emptyset}) = \tilde H_i(\abs{X}) = \tilde H_i(\abs{\skel^c_{-1} X})
\end{equation*}
which we assumed to be $0$.

Therefore, $\tilde H_i(\abs{\lk_X \sigma}) = 0$ for all faces $\sigma$ and all $i \geq r$, so $X$ is $r$-Leray.
\end{proof}

\subsection{Stacked balls}

A \emph{homology sphere} is a polytopal complex $X$ of dimension $d$ where for every face $\sigma$ including $\sigma = \emptyset$,
\begin{equation*}
\tilde H_i(\abs{\lk \sigma}) = \begin{cases*} 0 & if $i \neq d - \dim \sigma - 1$ \\ R & if $i = d - \dim \sigma - 1$. \end{cases*}
\end{equation*}
A \emph{homology manifold} is defined similarly but only considering $\sigma \neq \emptyset$. A \emph{homology ball} of dimension $d$ is a complex $X$ where:
\begin{itemize}
\item $\tilde H_i(\abs{X}) = 0$ for all $i$,
\item for every non-empty face $\sigma$,
\begin{equation*}
\tilde H_i(\abs{\lk \sigma}) = \begin{cases*} 0 & if $i \neq d - \dim \sigma - 1$ \\ \text{$R$ or $0$} & if $i = d - \dim \sigma - 1$, \end{cases*}
\end{equation*}
and
	\item the set of faces $\sigma$ with $\tilde H_{d-\dim \sigma - 1}(\abs{\lk \sigma}) = 0$ forms a subcomplex of $X$ that is a $(d-1)$-dimensional homology sphere.
\end{itemize}
If $\tilde H_{d-\dim \sigma - 1}(\abs{\lk \sigma})$ is $0$, $\sigma$ is called a \emph{boundary face}, and if this homology is $R$, $\sigma$ is an \emph{interior face}.

Suppose $X$ is a homology ball of dimension $d$. If every face of $X$ of dimension less than or equal to $d - s - 1$ is a boundary face, then $X$ is said to be \emph{$s$-stacked}. For example, if $X$ is $1$-stacked, then the interior faces must all have dimension $d$ or $d-1$; this condition is sometimes simply called ``stacked'' for simplicial complexes (e.g.\ \citet{art:Kalai-rigidity}), or ``capped'' for cubical complexes (e.g.\ \citet{art:Blind-Blind}). Simplicial $s$-stacked balls are well studied due to their connection to the Lower Bound Conjecture --- see e.g.\ \citet{art:Murai-Nevo}.

Just like Cohen--Macaulay complexes, we can characterise $s$-stacked balls using their co-skeletons.

\begin{thm} \label{thm:stacked-characterisation}
Suppose $X$ is a homology ball with dimension $d$. Then $X$ is $s$-stacked if and only if $\tilde H_{d-k-1}(\abs{\skel^c_k X}) = 0$ for all $k \leq d-s-1$ (or equivalently $\tilde H_j(\abs{\skel^c_{d-j-1} X}) = 0$ for all $j \geq s$, where $j = d-k-1$).
\end{thm}

\begin{proof}
By definition, a $k$-face $\sigma$ is a boundary face if and only if $\tilde H_{d-k-1}(\abs{\lk \sigma})$ is $0$. Therefore, $X$ is $s$-stacked if and only if
\begin{equation*}
\bigoplus_{\sigma \in X_k} \tilde H_{d-k-1}(\abs{\lk \sigma}) = 0
\end{equation*}
for all $k \leq d-s-1$.

We could apply this fact directly to the long exact sequence from \cref{thm:LES}; however, since $X$ is a homology ball, it is Cohen--Macaulay, so we can take a shortcut by using the short exact sequence from \cref{thm:CM-SES}:
\begin{equation*}
\begin{tikzcd}
0 \rar & \tilde H_{d-k}(\abs{\skel^c_{k-1} X}) \rar & \bigoplus_{\sigma \in X_k} \tilde H_{d-k-1}(\abs{\lk \sigma}) \rar & \tilde H_{d-k-1}(\abs{\skel^c_k X}) \rar & 0.
\end{tikzcd}
\end{equation*}
The inner term of this short exact sequence is $0$ for $k \leq d-s-1$ if and only if both outer terms are $0$ for the same range of $k$, which is equivalent to the claimed condition.
\end{proof}

\begin{rmk}
A closely related notion is an \emph{$s$-stacked sphere}, which is a complex that is the boundary of some $s$-stacked ball. Unfortunately, $s$-stacked spheres cannot be distinguished by topological features of their co-skeletons, at least for simplicial complexes. For example, the two spheres shown in \cref{fig:stacked-spheres} have homeomorphic $k$th co-skeletons for all $k$, but only one of the spheres is $1$-stacked as a simplicial complex.
\end{rmk}

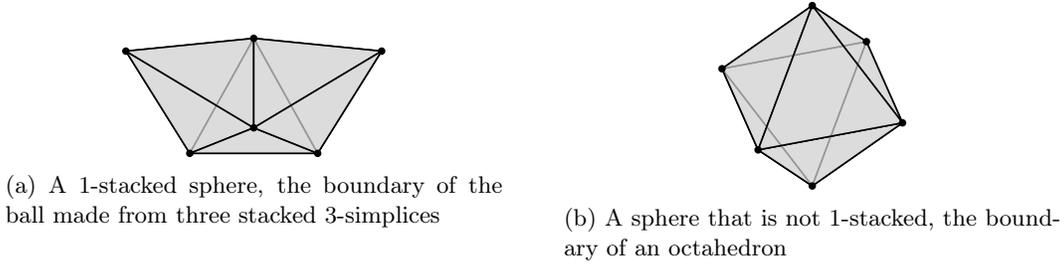
\begin{figure}
\centering
\begin{subfigure}{0.4\textwidth}
\centering
\begin{tikzpicture}[scale=1.7]

\coordinate (v1) at (-1,0.8);
\coordinate (v2) at (-0.5,0);
\coordinate (v3) at (0,0.2);
\coordinate (v4) at (0,0.9);
\coordinate (v5) at (0.5,0);
\coordinate (v6) at (1,0.8);

\filldraw [face] (v1)--(v2)--(v4)--cycle;
\filldraw [face] (v2)--(v4)--(v5)--cycle;
\filldraw [face] (v4)--(v5)--(v6)--cycle;
\filldraw [face] (v1)--(v2)--(v3)--cycle;
\filldraw [face] (v1)--(v3)--(v4)--cycle;
\filldraw [face] (v2)--(v3)--(v5)--cycle;
\filldraw [face] (v3)--(v4)--(v6)--cycle;
\filldraw [face] (v3)--(v5)--(v6)--cycle;

\foreach \i in {1,...,6}
	\node [vertex] at (v\i) {};

\end{tikzpicture}
\caption{A $1$-stacked sphere, the boundary of the ball made from three stacked $3$-simplices}
\end{subfigure} \qquad
\begin{subfigure}{0.4\textwidth}
\centering
\begin{tikzpicture}[3D, scale=1.2]

\coordinate (v1+) at (1,0,0);
\coordinate (v1-) at (-1,0,0);
\coordinate (v2+) at (0,1,0);
\coordinate (v2-) at (0,-1,0);
\coordinate (v3+) at (0,0,1);
\coordinate (v3-) at (0,0,-1);

\filldraw [face] (v1+)--(v2+)--(v3-) -- cycle;
\filldraw [face] (v1-)--(v2+)--(v3-) -- cycle;
\filldraw [face] (v1+)--(v2+)--(v3+) -- cycle;
\filldraw [face] (v1+)--(v2-)--(v3-) -- cycle;
\filldraw [face] (v1-)--(v2+)--(v3+) -- cycle;
\filldraw [face] (v1+)--(v2-)--(v3+) -- cycle;
\filldraw [face] (v1-)--(v2-)--(v3-) -- cycle;
\filldraw [face] (v1-)--(v2-)--(v3+) -- cycle;
\node [vertex] at (v1+) {};
\node [vertex] at (v1-) {};
\node [vertex] at (v2+) {};
\node [vertex] at (v2-) {};
\node [vertex] at (v3+) {};
\node [vertex] at (v3-) {};

\end{tikzpicture}
\caption{A sphere that is not $1$-stacked, the boundary of an octahedron}
\end{subfigure}
\caption{Two spheres with homeomorphic co-skeletons} \label{fig:stacked-spheres}
\end{figure}

\section{CAT(0) cubical complexes} \label{sec:CAT(0)}

One of the key features of a cubical complex is its \emph{hyperplanes}. If we associate an $r$-dimensional cube with the space $[0,1]^r$ in $\mathbb R^r$, then the \emph{$i$th hyperplane} of that cube is the subspace where $x_i = \frac{1}{2}$ --- see \cref{fig:iterated-hyperplanes-1st}. A hyperplane in a cubical complex $\Squelta$ is a maximally connected cubical complex obtained by glueing together hyperplanes of its component cubes where they meet along faces --- see \cref{fig:hyperplanes}. (We should note that in general, issues can arise when a hyperplane ``intersects itself'', in which case the hyperplane is not a true polytopal complex by our definitions. However, the focus of this section will be on CAT(0) cubical complexes, in which hyperplanes behave nicely --- see e.g.\ \cref{thm:CAT(0)-hyperplanes-from-cube-embedding} --- so we will not dwell on this issue.)

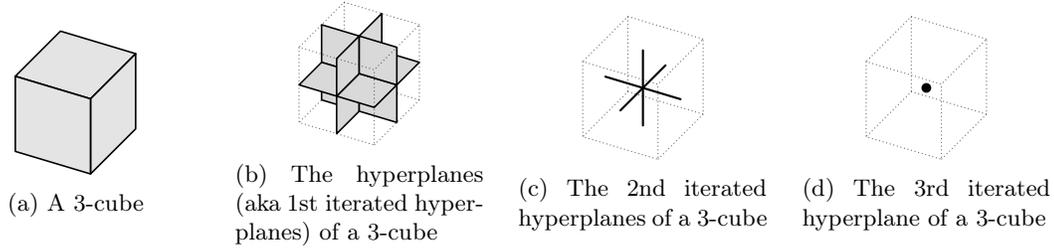
\begin{figure}
\centering
\begin{subfigure}{0.2\textwidth}
\centering
\begin{tikzpicture}[3D, baseline=0]

\filldraw [face] (0,0,1)--(1,0,1)--(1,1,1)--(0,1,1)--cycle;
\filldraw [face] (0,0,0)--(1,0,0)--(1,0,1)--(0,0,1)--cycle;
\filldraw [face] (0,0,0)--(0,1,0)--(0,1,1)--(0,0,1)--cycle;

\end{tikzpicture}
\caption{A $3$-cube} \label{fig:iterated-hyperplanes-0th}
\end{subfigure} \quad
\begin{subfigure}{0.2\textwidth}
\centering
\begin{tikzpicture}[3D, face/.append style={opacity=1}, baseline=0]

\draw [sketch edge] (1,0,0) -- (1,1,0) -- (0,1,0) (1,1,0) -- (1,1,1);

\filldraw [face] (0.5,0.5,0) -- (0.5,1,0) -- (0.5,1,0.5) -- (0.5,0.5,0.5) -- cycle;
\filldraw [face] (0,0.5,0) -- (0.5,0.5,0) -- (0.5,0.5,0.5) -- (0,0.5,0.5) -- cycle;
\filldraw [face] (0.5,0,0) -- (0.5,0.5,0) -- (0.5,0.5,0.5) -- (0.5,0,0.5) -- cycle;
\filldraw [face] (0.5,0,0.5) -- (1,0,0.5) -- (1,0.5,0.5) -- (0.5,0.5,0.5) -- cycle;
\filldraw [face] (0.5,0.5,0.5) -- (1,0.5,0.5) -- (1,0.5,1) -- (0.5,0.5,1) -- cycle;
\filldraw [face] (0.5,0.5,0.5) -- (0.5,1,0.5) -- (0.5,1,1) -- (0.5,0.5,1) -- cycle;
\filldraw [face] (0,0.5,0.5) -- (0,1,0.5) -- (0.5,1,0.5) -- (0.5,0.5,0.5) -- cycle;
\filldraw [face] (0,0,0.5) -- (0.5,0,0.5) -- (0.5,0.5,0.5) -- (0,0.5,0.5) -- cycle;
\filldraw [face] (0,0.5,0.5) -- (0.5,0.5,0.5) -- (0.5,0.5,1) -- (0,0.5,1) -- cycle;
\filldraw [face] (0.5,0,0.5) -- (0.5,0.5,0.5) -- (0.5,0.5,1) -- (0.5,0,1) -- cycle;

%

\draw [sketch edge] (0,0,0) -- (1,0,0) -- (1,0,1) -- (0,0,1) -- (0,1,1) -- (0,1,0) -- (0,0,0) -- (0,0,1) (1,0,1) -- (1,1,1) -- (0,1,1);

\end{tikzpicture}
\caption{The hyperplanes (aka $1$st iterated hyperplanes) of a $3$-cube} \label{fig:iterated-hyperplanes-1st}
\end{subfigure} \quad
\begin{subfigure}{0.2\textwidth}
\centering
\begin{tikzpicture}[3D, baseline=0]

\draw [sketch edge] (1,0,0) -- (1,1,0) -- (0,1,0) (1,1,0) -- (1,1,1);

\draw [edge, thick] (0.5,0.5,0)--(0.5,0.5,1);
\draw [edge, thick] (0.5,0,0.5)--(0.5,1,0.5);
\draw [edge, thick] (0,0.5,0.5)--(1,0.5,0.5);


\draw [sketch edge] (0,0,0) -- (1,0,0) -- (1,0,1) -- (0,0,1) -- (0,1,1) -- (0,1,0) -- (0,0,0) -- (0,0,1) (1,0,1) -- (1,1,1) -- (0,1,1);

\end{tikzpicture}
\caption{The $2$nd iterated hyperplanes of a $3$-cube} \label{fig:iterated-hyperplanes-2nd}
\end{subfigure} \quad
\begin{subfigure}{0.2\textwidth}
\centering
\begin{tikzpicture}[3D, baseline=0]

\draw [sketch edge] (1,0,0) -- (1,1,0) -- (0,1,0) (1,1,0) -- (1,1,1);

\node [vertex, inner sep=1.3pt] at (0.5,0.5,0.5) {};

\draw [sketch edge] (0,0,0) -- (1,0,0) -- (1,0,1) -- (0,0,1) -- (0,1,1) -- (0,1,0) -- (0,0,0) -- (0,0,1) (1,0,1) -- (1,1,1) -- (0,1,1);

\end{tikzpicture}
\caption{The $3$rd iterated hyperplane of a $3$-cube} \label{fig:iterated-hyperplanes-3rd}
\end{subfigure}
\caption{Hyperplanes and iterated hyperplanes of a cube} \label{fig:iterated-hyperplanes}
\end{figure}

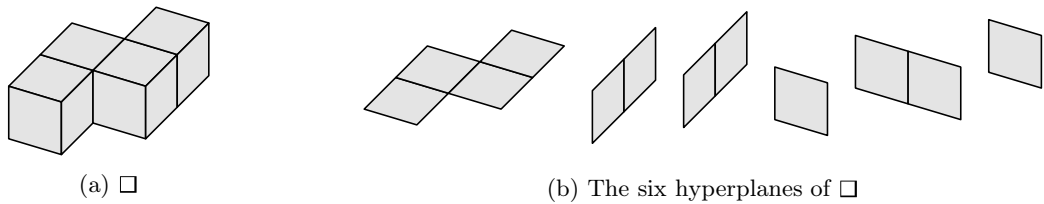
\begin{figure}
\centering
\begin{subfigure}{0.2\textwidth}
\centering
\begin{tikzpicture}[3D, scale=0.7, baseline=0]

\filldraw [face] (0,0,1) -- (1,0,1) -- (1,1,1) -- (0,1,1) -- cycle;
\filldraw [face] (1,0,1) -- (2,0,1) -- (2,1,1) -- (1,1,1) -- cycle;
\filldraw [face] (1,-1,1) -- (2,-1,1) -- (2,0,1) -- (1,0,1) -- cycle;
\filldraw [face] (0,0,0) -- (0,1,0) -- (0,1,1) -- (0,0,1) -- cycle;
\filldraw [face] (0,0,0) -- (1,0,0) -- (1,0,1) -- (0,0,1) -- cycle;
\filldraw [face] (1,-1,0) -- (1,0,0) -- (1,0,1) -- (1,-1,1) -- cycle;
\filldraw [face] (1,-1,0) -- (2,-1,0) -- (2,-1,1) -- (1,-1,1) -- cycle;
\filldraw [face] (2,-1,1) -- (2,0,1) -- (3,0,1) -- (3,-1,1) -- cycle;
\filldraw [face] (2,-1,0) -- (2,-1,1) -- (3,-1,1) -- (3,-1,0) -- cycle;


\end{tikzpicture}
\caption{$\Squelta$}
\end{subfigure} \quad
\begin{subfigure}{0.7\textwidth}
\centering
\begin{tikzpicture}[3D, baseline=0]

\matrix [cells={scale=0.7}, column sep=1em]
{
	\foreach \x/\y/\z in {0/0/0, 1/0/0, 1/-1/0, 2/-1/0}
		\filldraw [face] (\x,\y,\z+0.5) -- +(1,0,0) -- +(1,1,0) -- +(0,1,0) -- cycle;
	&
	\foreach \x/\y/\z in {0/0/0, 1/0/0}
		\filldraw [face] (\x,\y+0.5,\z) -- +(1,0,0) -- +(1,0,1) -- +(0,0,1) -- cycle;
	&
	\foreach \x/\y/\z in {1/-1/0, 2/-1/0}
		\filldraw [face] (\x,\y+0.5,\z) -- +(1,0,0) -- +(1,0,1) -- +(0,0,1) -- cycle;
	&
	\foreach \x/\y/\z in {0/0/0}
		\filldraw [face] (\x+0.5,\y,\z) -- +(0,1,0) -- +(0,1,1) -- +(0,0,1) -- cycle;
	&
	\foreach \x/\y/\z in {1/0/0, 1/-1/0}
		\filldraw [face] (\x+0.5,\y,\z) -- +(0,1,0) -- +(0,1,1) -- +(0,0,1) -- cycle;
	&
	\foreach \x/\y/\z in {2/-1/0}
		\filldraw [face] (\x+0.5,\y,\z) -- +(0,1,0) -- +(0,1,1) -- +(0,0,1) -- cycle;
	\\
};

\end{tikzpicture}
\caption{The six hyperplanes of $\Squelta$}
\end{subfigure}
\caption{A cubical complex and its hyperplanes} \label{fig:hyperplanes}
\end{figure}

\begin{figure}

\end{figure}

Since each hyperplane of a cubical complex is itself a cubical complex, it also has its own hyperplanes. Therefore, we will define an \emph{iterated hyperplane} of $\Squelta$ to be a hyperplane of a hyperplane of \dots of $\Squelta$. We will say the original hyperplanes are the ``$1$st iterated hyperplanes'', the hyperplanes of the hyperplanes are the ``$2$nd iterated hyperplanes'', and so on. See \crefrange{fig:iterated-hyperplanes-0th}{fig:iterated-hyperplanes-3rd}. As these figures show, an iterated hyperplane can sometimes be obtained in more than one way, when considered as a subspace of $\Squelta$ --- for example, each line segment in \cref{fig:iterated-hyperplanes-2nd} is a hyperplane of two different squares in \cref{fig:iterated-hyperplanes-1st}.

The following proposition reveals the connection between co-skeletons and iterated hyperplanes.

\begin{prop} \label{thm:cubical-coskeleton-simeq-iterated-hyperplanes}
Suppose $\Squelta$ is a cubical complex. Then $\abs{\bary(\Squelta) - \bary(\skel_k \Squelta)}$ is equal to the union of the $(k+1)$th iterated hyperplanes of $\Squelta$, considered as subspaces of $\Squelta$.
\end{prop}

For example, compare \cref{fig:iterated-hyperplanes-1st} and \cref{fig:bary-hyperplanes}. Note that in combination with \cref{thm:coskel-simeq-bary}, this proposition tells us that $\abs{\skel^c_k \Squelta}$ is homotopy equivalent to the union of the $(k+1)$th iterated hyperplanes.

\begin{proof}[Proof of \cref{thm:cubical-coskeleton-simeq-iterated-hyperplanes}]
We will show that the intersections of $\abs{\bary(\Squelta) - \bary(\skel_k \Squelta)}$ and of the union of the $(k+1)$th iterated hyperplanes with each cube of $\Squelta$ are the same.

\begin{figure}
\centering
\begin{tikzpicture}[3D, face/.append style={opacity=1}, scale=1.8]

\coordinate (v00*) at (0,0,0.5);
\coordinate (v01*) at (0,1,0.5);
\coordinate (v10*) at (1,0,0.5);
\coordinate (v11*) at (1,1,0.5);
\coordinate (v0*0) at (0,0.5,0);
\coordinate (v0*1) at (0,0.5,1);
\coordinate (v1*0) at (1,0.5,0);
\coordinate (v1*1) at (1,0.5,1);
\coordinate (v*00) at (0.5,0,0);
\coordinate (v*01) at (0.5,0,1);
\coordinate (v*10) at (0.5,1,0);
\coordinate (v*11) at (0.5,1,1);
\coordinate (v0**) at (0,0.5,0.5);
\coordinate (v1**) at (1,0.5,0.5);
\coordinate (v*0*) at (0.5,0,0.5);
\coordinate (v*1*) at (0.5,1,0.5);
\coordinate (v**0) at (0.5,0.5,0);
\coordinate (v**1) at (0.5,0.5,1);
\coordinate (v***) at (0.5,0.5,0.5);

\draw [sketch edge] (1,0,0) -- (1,1,0) -- (0,1,0) (1,1,0) -- (1,1,1);

\foreach \f/\e in {*10/**0, *10/*1*, 1*0/1**, 1*0/**0, 11*/1**, 11*/*1*}
	\filldraw [face] (v\f)--(v\e)--(v***)--cycle;
\foreach \f in {*10, 1*0, 11*}
	\node [vertex] at (v\f) {};

\foreach \f/\e in {*00/**0, *00/*0*, 0*0/**0, 0*0/0**, 01*/*1*, 01*/0**, *11/**1, *11/*1*, 1*1/**1, 1*1/1**, 10*/*0*, 10*/1**}
	\filldraw [face] (v\f)--(v\e)--(v***)--cycle;
\foreach \f in {**0, *1*, 1**}
	\node [vertex] at (v\f) {};

\foreach \f/\e in {00*/*0*, 00*/0**, 0*1/**1, 0*1/0**, *01/**1, *01/*0*}
	\filldraw [face] (v\f)--(v\e)--(v***)--cycle;
\draw [sketch edge] (0,0,0) -- (1,0,0) -- (1,0,1) -- (0,0,1) -- (0,1,1) -- (0,1,0) -- (0,0,0) -- (0,0,1) (1,0,1) -- (1,1,1) -- (0,1,1);
\foreach \f in {*00, 0*0, 01*, *11, 1*1, 10*, 00*, 0*1, *01, *0*, 0**, **1, ***}
	\node [vertex] at (v\f) {};

\end{tikzpicture}
\caption{$\bary([0,1]^3) - \bary(\skel_0 [0,1]^3)$} \label{fig:bary-hyperplanes}
\end{figure}

First, consider the union of the $(k+1)$th hyperplanes. In each cube, identified with $[0,1]^r$, a single $(k+1)$th iterated hyperplane is the result of setting $(k+1)$ of the coordinates to be $\frac{1}{2}$; therefore, the union of all $(k+1)$th iterated hyperplanes is the space
\begin{equation*}
Z \coloneqq \left\{ x \in \abs*{[0,1]^r} : \text{at least $k+1$ coordinates of $x$ are $\tfrac{1}{2}$} \right\}.
\end{equation*}

Now, consider $\abs{\bary(\Squelta) - \bary(\skel_k \Squelta)}$. The intersection of this space with some cube of $\Squelta$, identified with $\abs{[0,1]^r}$, is
\begin{equation*}
B \coloneqq \abs*{\bary([0,1]^r) - \bary(\skel_k [0,1]^r)}.
\end{equation*}
For each face $\sigma$ of $[0,1]^r$, we have a vertex $v_\sigma$ of $\bary([0,1]^r)$ at the barycenter of $\sigma$; the coordinates of $v_\sigma$ will all be $0$, $1$ or $\frac{1}{2}$, and the number of coordinates equal to $\frac{1}{2}$ is $\dim \sigma$. The remaining vertices of $\bary([0,1]^r)$ after deleting $\bary(\skel_k [0,1]^r)$, therefore, are precisely those which have at least $k+1$ coordinates equal to $\frac{1}{2}$.

For the rest of this proof, we will show that $B = Z$. First, we claim that $B \subseteq Z$. For each face $\{ v_{\sigma_1}, \dotsc, v_{\sigma_m} \}$ of $B$, the set $\{\sigma_1, \dotsc, \sigma_m\}$ is a chain, so it has a minimal element --- without loss of generality, say this is $\sigma_1$. By construction, the dimension of $\sigma_1$ is at least $k+1$, so there is some index set $I$ of at least $k+1$ elements so that $(v_{\sigma_1})_i = \frac{1}{2}$ for all $i \in I$. Since $\{\sigma_1, \dotsc, \sigma_m\}$ is a chain, every $v_{\sigma_j}$ for $j = 1, \dotsc, m$ must also satisfy $(v_{\sigma_j})_i = \frac{1}{2}$ for all $i \in I$. Therefore, the entire set $\{ v_{\sigma_1}, \dotsc, v_{\sigma_m} \}$ lies in the affine subspace $\{ x \in \mathbb R^r : \text{$x_i = \frac{1}{2}$ for all $i \in I$} \}$, so the convex hull of this set is contained in $Z$. Thus $B \subseteq Z$.

Conversely, let $z$ be a point in $Z$. By definition, there is an index set $I$ of size $k+1$ so that $z_i = \frac{1}{2}$ for all $i \in I$. The set of all points of $[0,1]^r$ satisfying $x_i = \frac{1}{2}$ for $i \in I$ is an iterated hyperplane of $[0,1]^r$, and its barycentric subdivision is the subcomplex of $\bary([0,1]^r) - \bary(\skel_k [0,1]^r)$ generated by vertices satisfying $(v_\sigma)_i = \frac{1}{2}$ for $i \in I$. Therefore, $z$ is in $B$, so $Z \subseteq B$.

In summary, the space $\abs{\bary(\Squelta) - \bary(\skel_k \Squelta)}$ and the union of the $(k+1)$th iterated hyperplanes agree within each cube of $\Squelta$, so they must be the same subspace.
\end{proof}

There is a class of cubical complexes where hyperplanes and co-skeletons work particularly well. In general, the link of any vertex in a cubical complex is isomorphic as a poset to the non-empty faces of a simplicial complex. A cubical complex is \emph{CAT(0)} if:
\begin{itemize}
	\item it is simply connected, and
	\item the link of every vertex is isomorphic to a \emph{flag} simplicial complex.
\end{itemize}
A simplicial complex is flag if the vertex set of every clique in its $1$-skeleton is the vertex set of a face.

CAT(0) cubical complexes were first studied from the perspective of metric spaces: equivalently, a cubical complex is CAT(0) if and only if it is simply connected and has non-positive curvature at every point (see \citet[Theorem~5.4 \& Theorem~5.18]{book:Bridson-Haefliger}; see also \citet{art:Gromov}). However, finite CAT(0) cubical complexes are also interesting from a purely combinatorial and topological perspective \citep{art:Ardila-geodesics, art:Hagen-hyperbolicity, art:Sageev, art:Roller, art:Rowlands-CAT0}.

One particular construction associated to a CAT(0) cubical complex is its \emph{crossing complex}, introduced in \citet{art:Rowlands-CAT0}: given a CAT(0) cubical complex $\Squelta$, its crossing complex, here denoted $\cross(\Squelta)$, is a simplicial complex whose vertices are in bijection with the hyperplanes of $\Squelta$, and a set of vertices forms a face if the intersection of the corresponding hyperplanes is non-empty. (In other words, the crossing complex is the nerve of the hyperplanes, as defined in \cref{thm:nerve}.)

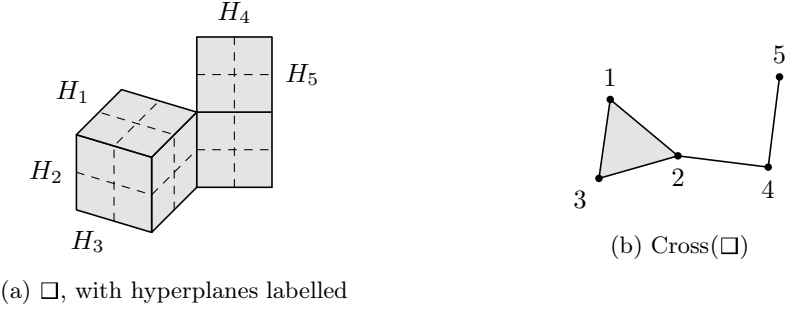
\begin{figure}
\centering
\begin{subfigure}{0.4\textwidth}
\centering
\begin{tikzpicture}

\filldraw [face, 3D] (0,0,0) -- (-1,0,0) -- (-1,0,1) -- (0,0,1) -- cycle;
\filldraw [face, 3D] (-1,0,0) -- (-1,1,0) -- (-1,1,1) -- (-1,0,1) -- cycle;
\filldraw [face, 3D] (-1,0,1) -- (0,0,1) -- (0,1,1) -- (-1,1,1) -- cycle;
\filldraw [face] (0,0) rectangle (1,1);
\filldraw [face] (0,1) rectangle (1,2);

\draw [hyperplane, 3D] (-0.5,0,0)--(-0.5,0,1)--(-0.5,1,1) node [label={[label distance=-5pt]above left:$H_1$}] {} (-1,0.5,0) node [label={[label distance=-5pt]below left:$H_3$}] {} --(-1,0.5,1)--(0,0.5,1) (-1,1,0.5) node [label={[label distance=-2pt]left:$H_2$}] {} --(-1,0,0.5)--(0,0,0.5);
\draw [hyperplane] (0,0.5)--(1,0.5) (0,1.5)--(1,1.5) node [label={[label distance=-2pt]right:$H_5$}] {} (0.5,0)--(0.5,2) node [label={[label distance=-2pt]above:$H_4$}] {};

\end{tikzpicture}
\caption{$\Squelta$, with hyperplanes labelled}
\end{subfigure}
\begin{subfigure}{0.4\textwidth}
\centering
\begin{tikzpicture}[scale=1.5]

\node [coordinate] (v1) at (0.1,0.7) {};
\node [coordinate] (v2) at (0.7,0.2) {};
\node [coordinate] (v3) at (0,0) {};
\node [coordinate] (v4) at (1.5,0.1) {};
\node [coordinate] (v5) at (1.6,0.9) {};

\filldraw [face] (v1)--(v2)--(v3)--(v1);
\draw [edge] (v2)--(v4)--(v5);

\node [vertex, label=above:$1$] at (v1) {};
\node [vertex, label=below:$2$] at (v2) {};
\node [vertex, label=below left:$3$] at (v3) {};
\node [vertex, label=below:$4$] at (v4) {};
\node [vertex, label=above:$5$] at (v5) {};

\end{tikzpicture}
\caption{$\cross(\Squelta)$}
\end{subfigure}
\caption{A CAT(0) cubical complex and its crossing complex}
\end{figure}

We record here some basic facts about CAT(0) cubical complexes and their crossing complexes.

\begin{lma}[{\citep[Corollary~1.5]{book:Bridson-Haefliger}, \citep[Section~4]{art:Gromov}}] \label{thm:CAT(0)-contractible}
Any CAT(0) cubical complex is contractible (not just simply connected).
\end{lma}

\begin{lma}[{\citep[Lemma~2.15]{art:Hagen-hyperbolicity}}] \label{thm:crossing-complex-flag}
$\cross(\Squelta)$ is always a flag simplicial complex. In particular, it is the clique complex of the ``crossing graph'' defined by Hagen \citet[Definition~2.16]{art:Hagen-hyperbolicity}.
\end{lma}

\begin{lma}[{\citep[Proposition~6.4, Corollary~6.5]{art:Rowlands-CAT0}}] \label{thm:CAT(0)-facets-bijection}
There is a bijection between the facets of $\Squelta$ and the facets of $\cross(\Squelta)$. If a facet of $\Squelta$ has dimension $k$, the corresponding facet of $\cross(\Squelta)$ has dimension $k-1$. As a consequence, $\dim \cross(\Squelta) = \dim \Squelta - 1$, and $\Squelta$ is pure if and only if $\cross \Squelta$ is pure.
\end{lma}

\begin{lma}[{\citep[Theorem~4.11]{art:Sageev}}] \label{thm:CAT(0)-hyperplanes-CAT(0)}
Every hyperplane (and thus every iterated hyperplane) in a CAT(0) cubical complex is also CAT(0).
\end{lma}

\citet[p.~9]{art:Ardila-geodesics} described a way to embed a CAT(0) cubical complex with $m$ hyperplanes into the cube $[0,1]^m$. This embedding is essentially canonical, up to the choice of a ``root vertex'' and a choice of labelling for the hyperplanes as $H_1, \dotsc, H_m$.

\begin{lma}[{\citep[Proposition~5.7]{art:Rowlands-CAT0}}] \label{thm:CAT(0)-hyperplanes-from-cube-embedding}
In this embedding of $\Squelta$ into $[0,1]^m$, the hyperplane $H_i$ of $\Squelta$ is the intersection of $\Squelta$ with the $i$th hyperplane of $[0,1]^m$.
\end{lma}

The resulting embedding of a hyperplane $H_i$ into the $i$th hyperplane of $[0,1]^m$ agrees with Ardila et al.'s canonical embedding of $H_i$ as a CAT(0) cubical complex into a cube; therefore, the iterated hyperplanes of $\Squelta$ are themselves each an intersection of $\Squelta$ with some set of hyperplanes of $[0,1]^m$. Therefore, if $S$ is a subset of $\{1, \dotsc, m\}$ of size $(k+1)$, we define
\begin{align*}
H_S & \coloneqq \Squelta \cap \bigcap_{i \in S} (\text{$i$th hyperplane of $[0,1]^m$}) \\
& = \bigcap_{i \in S} \big( \Squelta \cap (\text{$i$th hyperplane of $[0,1]^m$}) \big) = \bigcap_{i \in S} H_i. 
\end{align*}
Either this is a $(k+1)$th iterated hyperplane of $\Squelta$, or it is empty, and all $(k+1)$th iterated hyperplanes of $\Squelta$ are of this form. By definition of $\cross(\Squelta)$, $H_S$ is an iterated hyperplane if and only if $S$ is a face of $\cross(\Squelta)$; therefore, the $(k+1)$th iterated hyperplanes of $\Squelta$ correspond to $k$-faces of $\cross(\Squelta)$, for $k = 0, \dotsc, d-1$.

We can now begin to prove some new results about CAT(0) cubical complexes. The following proposition generalises \citet[Theorem~6.3]{art:Rowlands-CAT0}, which was the case of $k=0$.

\begin{prop} \label{thm:CAT(0)-coskeleton-simeq-crossing-coskeleton}
If $\Squelta$ is a $d$-dimensional CAT(0) cubical complex and $\Delta \coloneqq \cross(\Squelta)$ is its crossing complex, then $\abs{\skel^c_k \Squelta} \simeq \abs{\skel^c_{k-1} \Delta}$ for all $k = 0, \dotsc, d$.
\end{prop}

\begin{proof}
If $k = d$, then both $\abs{\skel^c_d \Squelta}$ and $\abs{\skel^c_{d-1} \Delta}$ are empty, since $\dim \Delta = d-1$ by \cref{thm:CAT(0)-facets-bijection}. Therefore, for the rest of this proof, assume that $k < d$.

We will apply the nerve theorem (\cref{thm:nerve}) to each of $\abs{\skel^c_k \Squelta}$ and $\abs{\skel^c_{k-1} \Delta}$, to show that they are both homotopy equivalent to the same nerve.

First, consider $\abs{\skel^c_k \Squelta}$, and define the family $\mathcal H \coloneqq \{ H_\sigma : \sigma \in \Delta_k \}$, where $H_\sigma \coloneqq \bigcap_{i \in \sigma} H_i$ as above. By \cref{thm:cubical-coskeleton-simeq-iterated-hyperplanes} and the above discussion, $\abs{\skel^c_k \Squelta}$ is homotopy equivalent to the union of $\mathcal H$. Each $H_\sigma$ is a closed set in $\Squelta$, and by \cref{thm:cubical-coskeleton-simeq-iterated-hyperplanes} their union is a triangulable space. Any intersection of spaces in $\mathcal H$ is itself an iterated hyperplane if it is non-empty, which is CAT(0) by \cref{thm:CAT(0)-hyperplanes-CAT(0)} and thus contractible by \cref{thm:CAT(0)-contractible}. Therefore, we may apply the nerve theorem to $\mathcal H$, and conclude that $\abs{\skel^c_k \Squelta} \simeq \abs{N(\mathcal H)}$.

But what is $N(\mathcal H)$? The vertices are in bijection with $k$-faces $\sigma$ of $\Delta$. A set of vertices, say $\{ \sigma_1, \dotsc, \sigma_m \}$, forms a face of $N(\mathcal H)$ whenever the intersection $H_{\sigma_1} \cap \dotsb \cap H_{\sigma_m}$ is non-empty. This intersection is the set
\begin{align*}
H_{\sigma_1} \cap \dotsb \cap H_{\sigma_m} & = \bigcap_{i \in \sigma_1} H_i \cap \dotsb \cap \bigcap_{i \in \sigma_m} H_i \\
& = \bigcap_{i \in \sigma_1 \cup \dotsb \cup \sigma_m} H_i \\
& = H_{\sigma_1 \cup \dotsb \cup \sigma_m}.
\end{align*}
So $\{ \sigma_1, \dotsc, \sigma_m \}$ is a face of $N(\mathcal H)$ if and only if $\sigma_1 \cup \dotsb \cup \sigma_m$ is a face of $\Delta$.

Now, let us turn our attention to $\skel^c_{k-1} \Delta$, the set of faces of $\Delta$ of dimension at least $k$. For each $k$-face $\sigma$ of $\Delta$, let $S_\sigma$ denote $\abs{\st_\Delta \sigma}$, and define $\mathcal S \coloneqq \{ S_\sigma : \sigma \in \Delta_k \}$. Every face of $\Delta$ of dimension at least $k$ is contained in $\st_\Delta \sigma$ for some $k$-face $\sigma$, and no face of dimension less than $k$ is, so the union of $\mathcal S$ is $\abs{\skel^c_{k-1} \Delta}$. Each $S_\sigma$ is open in $\Delta$, and any non-empty intersection of stars is itself a star and thus contractible. Therefore, we may apply the nerve theorem to $\mathcal S$ as well.

Let us examine $N(\mathcal S)$. The vertices are in bijection with $k$-faces $\sigma$ of $\Delta$, and a set $\{ \sigma_1, \dotsc, \sigma_m \}$ forms a face of $N(\mathcal S)$ whenever $\st_\Delta \sigma_1 \cap \dotsb \cap \st_\Delta \sigma_m$ is non-empty. This happens if and only if there exists a face of $\Delta$ that contains all of $\sigma_1, \dotsc, \sigma_m$, which happens whenever $\sigma_1 \cup \dotsb \cup \sigma_m$ is a face of $\Delta$.

But this means that $N(\mathcal S) = N(\mathcal H)$! Therefore,
\begin{equation*}
\abs{\skel^c_k \Squelta} \simeq \abs{N(\mathcal H)} = \abs{N(\mathcal S)} \simeq \abs{\skel^c_{k-1} \Delta}. \qedhere
\end{equation*}
\end{proof}

Next, we will examine some applications of this result, but first we need one more construction.

There is a standard bijection from the non-empty faces $\sigma$ of the cube $[0,1]^r$ to vectors $\chi^\sigma$ in $\{0,1,*\}^r$, where the $i$th coordinate of $\chi^\sigma$ is $0$ or $1$ if all points $x$ in $\sigma$ have $x_i = 0$ or $x_i = 1$ respectively, and $\chi^\sigma_i$ is $*$ otherwise. If $\sigma$ and $\tau$ are two faces, then $\sigma \subseteq \tau$ if and only if $\chi^\sigma_i \preceq \chi^\tau_i$ for all $i$, where $0 \preceq *$ and $1 \preceq *$ are the only non-equality relations on $0$, $1$ and $*$.

Suppose $\Delta$ is an abstract simplicial complex with vertex set $\{1, \dotsc, n\}$. We define the \emph{cubical cone} over $\Delta$, denoted $\cone(\Delta)$, to be the subcomplex of $[0,1]^n$ consisting of all faces $\sigma$ such that the set $\big\{ i : \chi^\sigma_i \in \{1,*\} \big\}$ is a face of $\Delta$ (including the empty face). For example, see \cref{fig:cubical-cone}. The reason for the name ``cubical cone'' is the following lemma, analogous to properties of a simplicial cone.

\begin{lma}
The cubical cone over $\Delta$ is the unique subcomplex of a cube (up to isomorphism) with the property that there is a vertex $v_0$ (the ``cone point'', which is the point $(0,\dotsc, 0)$ in the standard labelling) such that:
\begin{itemize}
	\item every facet contains $v_0$, and
	\item the link of $v_0$ is (isomorphic as a poset to the non-empty faces of) $\Delta$.
\end{itemize}
\end{lma}

\begin{proof}
First, let us check that the cubical cone does satisfy these properties. If $\sigma$ is a face of $\cone(\Delta)$, we can replace every $1$ in $\chi^\sigma$ with $*$ to obtain the vector of a face of $\cone(\Delta)$ that contains $\sigma$; therefore, the vectors of facets of $\cone(\Delta)$ consist only of $0$ and $*$, so every facet contains $v_0 = (0,\dotsc,0)$. Conversely, the faces that properly contain $v_0$ are precisely the faces $\sigma \neq v_0$ such that $\chi^\sigma$ consists of $0$ and $*$, and these faces are in bijection with non-empty faces of $\Delta$, so the link of $v_0$ is indeed $\Delta$.

Now, suppose $\Squelta$ is an arbitrary subcomplex of a cube satisfying these properties. By the symmetry of the cube, we may assume that $v_0$ is the vertex $(0,\dotsc,0)$. As before, the link of $v_0$ is then the set of faces of $\Squelta$ whose vectors consist only of $0$ and $*$; since we assume this set is isomorphic to $\Delta$, we may again use the symmetry of the cube to permute the coordinates and assume that this set is equal to $\lk_{\cone(\Delta)} v_0$. Since $\Squelta$ contains these faces of the cube, it must also contain every sub-face of these faces, so $\Squelta$ contains every face of $\cone(\Delta)$. Moreover, it cannot contain any other faces: any such face must be contained in some facet of $\Squelta$, but that facet cannot be one that we have already accounted for in $\lk_\Squelta v_0$, so the facet would not contain $v_0$. Therefore, $\Squelta \cong \cone(\Delta)$.
\end{proof}

While every flag simplicial complex is the crossing complex of some CAT(0) cubical complex (see \citet[Proposition~2.19]{art:Hagen-hyperbolicity} and \citet[Lemma~3.3]{art:Rowlands-CAT0}, where the construction is the cubical cone), in general, there can be many CAT(0) cubical complexes with the same crossing complex. However, we can now prove that if $\cross(\Squelta)$ is a connected manifold, the only possibility is $\Squelta = \cone(\Delta)$.

\begin{figure}
\centering
\begin{subfigure}{0.3\textwidth}
\centering
\begin{tikzpicture}[scale=1.2]

\foreach \x/\y [count=\i] in {0.5/1.8, -0.1/0.9, 1.0/1.1, 1.2/0.2, 0.0/0.0}
	\coordinate (v\i) at (\x,\y);

\filldraw [face] (v1)--(v2)--(v3)--cycle;
\draw [edge] (v3)--(v4)--(v5)--(v2);

\foreach \pos [count=\i] in {above, left, above right, below right, below left}
	\node [vertex, label=\pos:$\i$] at (v\i) {};

\end{tikzpicture}
\caption{$\Delta$}
\end{subfigure}
\begin{subfigure}{0.5\textwidth}
\centering
\begin{tikzpicture}[3D, scale=1.5]

\foreach \x/\y/\z in {0/0/0, 0/1/0, 1/0/0, 1/1/1}
	\filldraw [face] (\x,\y,\z) -- +(1,0,0) -- +(1,1,0) -- +(0,1,0) -- cycle;
\filldraw [face] (1,1,0) -- (2,1,0) -- (2,1,1) -- (1,1,1) -- cycle;
\filldraw [face] (1,1,0) -- (1,2,0) -- (1,2,1) -- (1,1,1) -- cycle;

\foreach \x/\y/\z/\coord/\pos in {1/1/0/v_0 = 00000/above, 1/1/1/10000/above, 1/2/0/01000/above left, 2/1/0/00100/above right, 1/0/0/00010/below right, 0/1/0/00001/below left, 1/2/1/11000/above left, 2/1/1/10100/above right, 2/0/0/00110/right, 0/0/0/00011/below, 0/2/0/01001/left, 2/2/1/11100/above}
	\node [vertex, label={[fill=white, fill opacity=0.7, text opacity=1]\pos:{\small$\coord$}}] at (\x,\y,\z) {};

\end{tikzpicture}
\caption{The cubical cone over $\Delta$ (embedded into $\mathbb R^3$), with vertices labelled}
\end{subfigure}
\caption{A cubical cone} \label{fig:cubical-cone}
\end{figure}
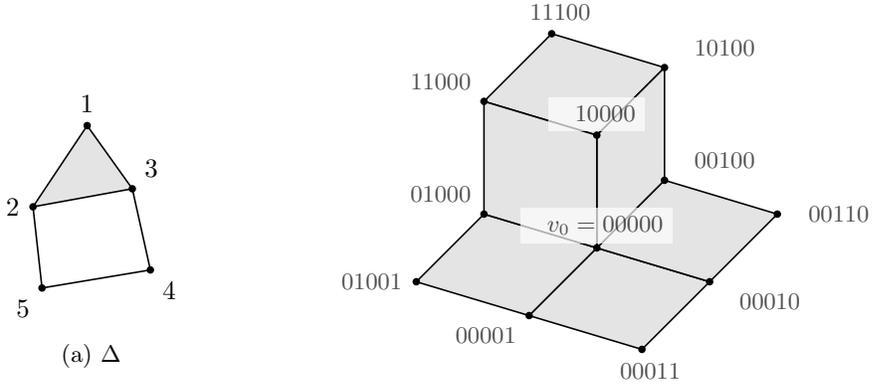

\begin{crl}
Let $\Squelta$ be a $d$-dimensional CAT(0) cubical complex and $\Delta$ its crossing complex. If $\Delta$ is a connected homology manifold, then $\Squelta$ is the cubical cone over $\Delta$.
\end{crl}

\begin{proof}
Examine this section of the long exact sequence from \cref{thm:LES}, taking $k = 0$ and using coefficients in $\mathbb Z/2$:
\begin{equation*}
\begin{tikzcd}
\tilde H_d(\abs{\skel^c_{-1} \Squelta}) \rar & \bigoplus_{v \in \Squelta_0} \tilde H_{d-1}(\abs{\lk v}) \rar & \tilde H_{d-1}(\abs{\skel^c_0 \Squelta}) \rar & \tilde H_{d-1}(\abs{\skel^c_{-1} \Squelta}).
\end{tikzcd}
\end{equation*}
Note that $\abs{\skel^c_{-1} \Squelta}$ is simply $\abs{\Squelta}$, which is contractible since $\Squelta$ is CAT(0), so the first and last terms of this sequence are $0$. By \cref{thm:CAT(0)-coskeleton-simeq-crossing-coskeleton}, $\tilde H_{d-1}(\abs{\skel^c_0 \Squelta})$ is isomorphic to $\tilde H_{d-1}(\abs{\skel^c_{-1} \Delta}) = \tilde H_{d-1}(\abs{\Delta})$, and since $\Delta$ is a connected homology manifold of dimension $d-1$, this homology group is $\mathbb Z/2$. Therefore,
\begin{equation*}
\bigoplus_{v \in \Squelta_0} \tilde H_{d-1}(\abs{\lk v}) \cong \mathbb Z/2.
\end{equation*}
Thus there must be one vertex $v_0$ for which $\tilde H_{d-1}(\abs{\lk v_0}) = \mathbb Z/2$, and all other vertices have $\tilde H_{d-1}(\abs{\lk v}) = 0$.

Every vertex link in a CAT(0) cubical complex is an induced subcomplex of the crossing complex \citep[Lemma~4.9]{art:Rowlands-CAT0}. However, the only induced subcomplex of a connected $(d-1)$-manifold with non-zero homology in degree $d-1$ is the entire manifold. Therefore, $\lk v_0 = \Delta$.

\Cref{thm:CAT(0)-facets-bijection} says that there is a bijection between the facets of the cubical complex and the facets of its crossing complex. In this case, since $\lk v_0 = \Delta$, we also have a bijection between the facets of the crossing complex $\Delta$ and facets of $\Squelta$ that contain $v_0$. Therefore, every facet of $\Squelta$ must contain $v_0$. Hence $\Squelta$ is a cubical cone.
\end{proof}

We conclude this paper by combining \cref{thm:CAT(0)-coskeleton-simeq-crossing-coskeleton} with \cref{thm:CM-characterisation,thm:stacked-characterisation,thm:Leray-characterisation}, to compare the combinatorial properties discussed in \cref{sec:families} for $\Squelta$ and $\cross(\Squelta)$.

\begin{thm} \label{thm:CAT(0)-property-iff-crossing}
Suppose $\Squelta$ is a CAT(0) cubical complex and $\Delta$ is its crossing complex. Then:
\begin{enumerate}[(a)]
	\item \label{thm:CAT(0)-CM-iff-crossing} $\Squelta$ is Cohen--Macaulay if and only if $\Delta$ is Cohen--Macaulay.
	\item \label{thm:CAT(0)-Leray-iff-crossing} $\Squelta$ is $r$-Leray if and only if $\Delta$ is $r$-Leray.
	\item \label{thm:CAT(0)-stacked-iff-crossing} If $\Squelta$ and $\Delta$ are both homology balls, then $\Squelta$ is $s$-stacked if and only if $\Delta$ is $s$-stacked.
\end{enumerate}
\end{thm}

Note that the assumption that both $\Squelta$ and $\Delta$ are balls in \ref{thm:CAT(0)-stacked-iff-crossing} is essential: it is possible for one to be a ball and the other not. For example, if $\Squelta$ consists of $m$ squares glued in a $1 \times m$ rectangle, it is a $1$-stacked ball, but its crossing complex is a star graph with $m$ pendants, which is not a ball for $m > 2$.

\begin{proof}[Proof of \cref{thm:CAT(0)-property-iff-crossing}]
\begin{enumerate}[(a)]
	\item Let $d = \dim \Squelta$. By \cref{thm:CM-characterisation}, $\Squelta$ is Cohen--Macaulay if and only if it is pure and $\tilde H_i(\abs{\skel^c_k \Squelta}) = 0$ for all $k = -1, \dotsc, d$ and all $i < d-k-1$. Any CAT(0) cubical complex is always contractible, so $\tilde H_i(\abs{\skel^c_{-1} \Squelta})$ is guaranteed to be $0$, hence we only need to consider $k = 0, \dotsc, d$. Now, \cref{thm:CAT(0)-coskeleton-simeq-crossing-coskeleton} says that $\tilde H_i(\abs{\skel^c_k \Squelta}) = \tilde H_i(\abs{\skel^c_{k-1} \Delta})$ for $k$ in this range, and \cref{thm:CAT(0)-facets-bijection} says that $\Squelta$ is pure if and only if $\Delta$ is pure; therefore, $\Squelta$ is Cohen--Macaulay if and only if $\Delta$ is pure and $\tilde H_i(\abs{\skel^c_{k-1} \Delta}) = 0$ for all $k = 0, \dotsc, d$ and all $i < d - k - 1 = (d-1) - (k-1) - 1$. But since $\dim \Delta = d-1$, this is exactly the condition given by \cref{thm:CM-characterisation} for $\Delta$ to be Cohen--Macaulay.
	\item By \cref{thm:Leray-characterisation}, $\Squelta$ is $r$-Leray if and only if $\tilde H_i(\abs{\skel^c_k \Squelta}) = 0$ for all $i \geq r$ and all $k = -1, \dotsc, d$. Because $\Squelta$ is CAT(0), $\abs{\skel^c_{-1} \Squelta}$ is guaranteed to be contractible, so we only need to check $k = 0, \dotsc, d$. By \cref{thm:CAT(0)-coskeleton-simeq-crossing-coskeleton}, this condition is equivalent to $\tilde H_i(\abs{\skel^c_j \Delta})$ being $0$ for all $i \geq r$ and all $j = -1, \dotsc, d-1$, where $j = k-1$. And this precisely means that $\Delta$ is $r$-Leray.
	\item By \cref{thm:stacked-characterisation}, $\Squelta$ is $s$-stacked if and only if $\tilde H_{d-k-1}(\abs{\skel^c_k \Squelta}) = 0$ for all $k \leq d-s-1$, where $d = \dim \Squelta$. \Cref{thm:CAT(0)-coskeleton-simeq-crossing-coskeleton} tells us this is equivalent to $\tilde H_{d-k-1}(\abs{\skel^c_{k-1} \Delta})$ being $0$ for $k \leq d-s-1$. Equivalently, letting $j = k-1$, we have $\tilde H_{(d-1)-j-1}(\abs{\skel^c_j \Delta}) = 0$ for $j \leq (d-1)-s-1$. Since $\dim \Delta = d-1$, this is equivalent to $\Delta$ being $s$-stacked. \qedhere
\end{enumerate}
\end{proof}

\bibliographystyle{abbrvnat}
\bibliography{coskeletons-refs}

\end{document}